\documentclass[reqno]{amsart}

\usepackage{amsmath}

\usepackage{amsthm}
\usepackage{amssymb}
\usepackage{graphics}
\usepackage{latexsym}
\usepackage{comment}
\usepackage{extarrows}
\usepackage{wrapfig}
\usepackage{comment}
\usepackage {colortbl,array,xcolor}
\usepackage{hhline}
\usepackage{tikz}
\usepackage{enumitem}
\usetikzlibrary{intersections, calc, math, patterns}

\newtheorem{theorem}{Theorem}[section]
\newtheorem{lemma}[theorem]{Lemma}
\newtheorem{proposition}[theorem]{Proposition}

\newtheorem{corollary}[theorem]{Corollary}
\theoremstyle{definition}
\newtheorem{definition}[theorem]{Definition}

\theoremstyle{remark}
\newtheorem{remark}{Remark}

\newcommand{\R}{{\mathbb R}}
\newcommand{\Z}{{\mathbb Z}}
\newcommand{\T}{{\mathbb T}}

\newcommand{\N}{{\mathbb N}}
\newcommand{\C}{{\mathbb C}}

\newcommand{\F}{\mathcal{F}}

\newcommand{\supp}{\operatorname{supp}}

\numberwithin{equation}{section}

\def\1{\textbf{\rm 1}}

\def\XXint#1#2#3{{\setbox0=\hbox{$#1{#2#3}{\int}$}
\vcenter{\hbox{$#2#3$}}\kern-.5\wd0}}

\newcommand{\vertiii}[1]{{\left\vert\kern-0.25ex\left\vert\kern-0.25ex\left\vert #1 
\right\vert\kern-0.25ex\right\vert\kern-0.25ex\right\vert}}

\title[Shell decoupling and periodic Zakharov system]{
Decoupling inequality for paraboloid under shell type restriction and its application to the periodic Zakharov system}

\author[S.~Kinoshita]{Shinya Kinoshita}
\address[S.~Kinoshita]{Department of Mathematics, Tokyo Institute of Technology, Meguro-ku, Tokyo, 152-8551, Japan}
\email[S.~Kinoshita]{kinoshita@math.titech.ac.jp}

\author[S.~Nakamura]{Shohei Nakamura}
\address[S.~Nakamura]{Department of Mathematics, Graduate School of Science, Osaka University, Toyonaka, Osaka 560-0043, Japan}
\email[S.~Nakamura]{srmkn@math.sci.osaka-u.ac.jp}

\author[A.~Sanwal]{Akansha Sanwal}
\address[A.~Sanwal]{Institut f\"{u}r Mathematik, Universit\"{a}t Innsbruck, Technikerstrasse 13, 6020 Innsbruck, Austria.}
\email[A.~Sanwal]{akansha.sanwal@uibk.ac.at}

\begin{document}

\maketitle

\begin{abstract}
In this paper, we establish local well-posedness for the Zakharov system on $\mathbb{T}^d$, $d\ge3$ in a low regularity setting. Our result  improves the work of Kishimoto \cite{Kishi13}. Moreover, the result is sharp up to $\varepsilon$-loss of regularity when $d=3$ and $d\ge5$ as long as one utilizes the iteration argument. 
We introduce ideas from recent developments of the Fourier restriction theory. 
The key element in the proof of our well-posedness result is a new trilinear discrete Fourier restriction estimate involving paraboloid and cone. 
We prove this trilinear estimate by improving  Bourgain--Demeter's range of exponent for the linear decoupling inequality for paraboloid \cite{BourDem15} under the constraint that the input space-time function $f$ satisfies 
${\rm supp}\, \hat{f} 
\subset 
\{
(\xi,\tau) \in \mathbb{R}^{d+1}:
1- \frac1N \le |\xi| \le 1 + \frac1N,\; 
|\tau - |\xi|^2| \le \frac1{N^{2}}
\}
$ for large $N\ge1$. 
\end{abstract}

\section{Introduction}\label{sec:intro}
\subsection{Introduction and main result}
We consider the Cauchy problem for the Zakharov system:
\begin{equation}\label{Zakharov}
\begin{cases}
i \partial_t u  + \Delta u = nu , \\
\partial_t^2 n -\Delta n  = \Delta(|u|^2),\\
(u(0), n(0), \partial_t n(0)) = (u_0, n_0, n_1),
\end{cases}
\end{equation}
where the unknown functions $u$ and $n$ are complex- and real-valued functions, respectively. We consider \eqref{Zakharov} with periodic boundary condition, namely, $u:  \mathbb{R} \times \mathbb{T}^d :\to \C$ and $n:  \mathbb{R} \times \mathbb{T}^d: \to \R$ where $\T^d = (\R/2 \pi \Z)^d$. 
We take initial data from $L^2$-based Sobolev spaces:
\[
(u_0, n_0, n_1) \in  H^s(\T^d) \times H^{\ell}(\T^d) \times H^{\ell-1}(\T^d) =: \mathcal{H}^{s,\ell}(\T^d)
\]
with $s, \ell \in \R$.

The Zakharov system was derived by Zakharov \cite{Zakharov72} to describe the propagation of Langmuir waves in a plasma. 
The Zakharov system possesses two conservation laws, namely, the conservation of mass $\|u(t)\|_{L^2} = \|u_0\|_{L^2}$ and energy:
\[
E(u,n)(t) = \|\nabla u(t)\|_{L^2}^2 + \frac{1}{2} (\|n(t)\|_{L^2}^2 + \| |\nabla|^{-1} \partial_t n(t)\|_{L^2}^2) + 
\int_{\T^d} n(t,x) |u(t,x)|^2 dx.
\]
Here, if the initial velocity has zero mean, i.e. $\widehat{n_1} (0)=0$, then the energy is well-defined. See the introduction in \cite{KM13} for more details. 

The Cauchy problem for the Zakharov system on $\R \times \R^d$ has been studied extensively. 
In \cite{OT92}, Ozawa--Tsutsumi obtained the local well-posedness of \eqref{Zakharov} in $H^2(\R^d) \times H^1(\R^d) \times L^2(\R^d)$ for $d \leq 3$. 
After the seminal works by Bourgain \cite{BourGAFA93}, and Klainerman--Machedon \cite{KM93}, Bourgain--Colliander \cite{BC96} applied the Fourier restriction norm method to the study of \eqref{Zakharov} and obtained global well-posedness in $H^1(\R^d) \times L^2(\R^d) \times H^{-1}(\R^d)$ for $d =2$, $3$. 
Thereafter, Ginibre--Tsutsumi--Velo \cite{GTV97} extended the range of $(s,\ell)$ for which the Zakharov system is locally well-posed as follows:
\[
\begin{cases}
d=1, \qquad &-\frac12 \leq s-\ell \leq 1, \ 2 s \geq \ell + \frac12 \geq 0,\\
d=2,3, & \ell \leq s \leq \ell +1, \ \ell \geq 0, \ 2s-(\ell+1) \geq 0,\\
d \geq 4, & \ell \leq s \leq \ell +1, \ \ell > \frac{d}{2} - 2, \ 2 s -(\ell +1) > \frac{d}{2}-2.
\end{cases}
\]
They also introduced the notion of scaling criticality for \eqref{Zakharov} using which the scaling critical values are given by $(s,\ell) = (\frac{d-3}{2}, \frac{d-4}{2})$. 
In the $2$-dimensional case, on the line $\ell = s - \frac12$, Bejenaru--Herr--Holmer--Tataru \cite{BHHT09} pushed down the threshold of necessary regularity by $\frac12$ by proving that \eqref{Zakharov} is locally well-posed in $L^2(\R^2) \times H^{-\frac12}(\R^2) \times H^{-\frac32}(\R^2)$. 
They employed the nonlinear version of the Loomis--Whitney inequality which has a connection with multilinear restriction estimates, see e.g. \cite{BCW05, BCT, BHT10}. 
For $d=3$, Bejenaru--Herr \cite{BH11} proved local well-posedness in $H^s(\R^3) \times H^{\ell}(\R^3) \times H^{\ell-1}(\R^3)$ with $\ell > -\frac12$, $\ell \leq s \leq \ell +1$, $2 s  > \ell + \frac12$. 
In the energy critical dimension $d=4$, Bejenaru--Guo--Herr--Nakanishi \cite{BGHN} established the small data global well-posedness and scattering in the energy space $H^1(\R^4) \times L^2(\R^4) \times \dot{H}^{-1}(\R^4)$. Recently, under a radially symmetric assumption on data, global well-posedness and scattering in $4$D energy space below the ground state was shown by Guo--Nakanishi \cite{GN21}. Without the radial symmetry assumption, global well-posedness in $4$D energy space below the ground state has been established in \cite{CHN21}. However, scattering below ground state in the non-radial energy space remains a challenging open problem. 
In the recent paper \cite{CHN}, the sharp range of $(s, \ell)$ for well-posedness is determined for spatial dimensions $d \geq 4$. 
For $d\leq 3$, we also refer to the recent results \cite{CW21} and \cite{San22} for the almost optimal range of regularity for the local well-posedness of \eqref{Zakharov}.\\

On the other hand, there are not so many works on the periodic Zakharov system \eqref{Zakharov}. 
For $\gamma >0$, let $\T_{\gamma} = (\R / 2 \pi \gamma \N)$. 
For $d=1$, it is known that the necessary regularity for the well-posedness depends on the period of the torus. 
Bourgain \cite{Bour94} and Takaoka \cite{Taka99} proved that \eqref{Zakharov} is locally well-posed in $L^2(\T_{\gamma}) \times H^{-\frac{1}{2}}(\T_{\gamma}) \times H^{-\frac32}(\T_{\gamma})$ if $\gamma \notin \N$ and $H^{\frac12}(\T_{\gamma}) \times L^2(\T_{\gamma}) \times H^{-1}(\T_{\gamma})$ if $\gamma \in \N$. 
Notice that their results, in some sense, are sharp. This is due to the fact that the problematic nonlinear interaction arises only in the case $\gamma \in \N$. 
For $d=1$, we also note that global well-posedness of \eqref{Zakharov} in the energy space is known \cite{Bour94}. 

The most recent result on the local well-posedness of \eqref{Zakharov} on the multidimensional torus is due to Kishimoto~\cite{Kishi13}. 
He proved the following:
\begin{theorem}[Theorem 1.1 in \cite{Kishi13}]\label{theorem:Kishimoto13}
Let
\[
\begin{cases}
0 \leq s -\ell \leq 1, \ 2 s \geq \ell + \frac{d}{2} > d-1, & (d \geq 3),\\
0 \leq s -\ell \leq 1, \ 2 s \geq \ell + \frac{d}{2} \geq 1, & (d =2).
\end{cases}
\]
Then \eqref{Zakharov} is locally well-posed in $\mathcal{H}^{s,\ell} (\T^d)$.
\end{theorem}
In particular, for the case $d=2$, local well-posedness in the energy space $H^1(\T^2) \times L^2(\T^2) \times H^{-1}(\T^2)$ was obtained. 
This readily implies the global well-posedness under the conditions $\widehat{n}_1(0)=0$ and $\|u_0\|_{L^2} \ll 1$. 
Moreover, the result remains true if the classical torus $\T^d$ is replaced by $\T_{\gamma}^d := (\R/2 \pi \gamma_1 \Z) \times \cdots \times (\R / 2 \pi \gamma_d \Z)$ for $\gamma = (\gamma_1, \ldots, \gamma_d)$ with $\gamma_j>0$ for each $j=1,\ldots, d$. 
In \cite{KM13}, Kishimoto--Maeda succeeded in removing the condition $\widehat{n}_1(0)=0$ and found that the smallness condition $\|u_0\|_{L^2} \ll 1$ can be replaced with $\|u_0\|_{L^2} \leq \|Q\|_{L^2}$. 
By $Q$, we denote the ground state solution of the focusing cubic nonlinear Schr\"odinger equation on $\R^2$, that is, $Q$ is the unique positive radial solution of
\[
- \Delta Q + Q - Q^3=0, \quad x \in \R^2.
\]

The scaling critical regularity $(s,\ell) = (\frac{d-3}{2}, \frac{d-4}{2})$ is on the line $\ell = s - \frac12$. 
Our aim is to relax the regularity threshold in Theorem \ref{theorem:Kishimoto13} for $d \geq 3$ on the line $\ell= s - \frac12$. 
We note that for $\ell = s - \frac12$, the condition in Theorem \ref{theorem:Kishimoto13} becomes $s > \frac{d-1}{2}$ when $d \geq 3$. 
We now state our main result.
\begin{theorem}\label{theorem:main}
Let $d \geq 3$. Define
\begin{equation}\label{assumption:regularity}
s_0 = 
\begin{cases}
\frac12 \quad & (d=3),\\
\frac34 & (d=4),\\
\frac{d-3}{2} & (d\geq 5).
\end{cases}
\end{equation}
Let $s > s_0$. Then~\eqref{Zakharov} is locally well-posed in $\mathcal{H}^{s,s-\frac12}(\T^d)$.
\end{theorem}
\begin{remark}
Similar to Kishimoto's result, we can replace the classical torus $\T^d$ with $\T_{\gamma}^d$. 
\end{remark}

We observe that the regularity threshold \eqref{assumption:regularity} cannot be improved as far as we utilize the iteration argument in the cases $d=3$ and $d \geq 5$. 
Precisely, for $d \geq 2$ Kishimoto proved that if $s < \frac12$ then, for any $T>0$, the data-to-solution map $(u_0,n_0,n_1) \mapsto (u,n, \partial_t n)$
of \eqref{Zakharov}, as a map from the unit ball in
$\mathcal{H}^{s,s-\frac12}(\T^d)$ to
$C([0,T]; \mathcal{H}^{s,s-\frac12}(\T^d))$ fails to be $C^2$-differentiable at the origin, see Theorem 1.2 in \cite{Kishi13}. 
This suggests that Theorem \ref{theorem:main} in the case $d=3$ is optimal up to $\varepsilon$ loss. 

In addition, because $(s,\ell) = (\frac{d-3}{2}, \frac{d-4}{2})$ is the scaling critical regularity, the standard argument yields the sharpness of Theorem \eqref{theorem:main} for the case $d \geq 5$ up to $\varepsilon$ loss. 
\begin{theorem}\label{theorem:NotC^2}
Let $s < \frac{d-3}{2}$. Then, for any $T>0$, the data-to-solution map $(u_0,n_0,n_1) \mapsto (u,n, \partial_t n)$
of \eqref{Zakharov}, as a map from the unit ball in
$\mathcal{H}^{s,s-\frac12}(\T^d)$ to
$C([0,T]; \mathcal{H}^{s,s-\frac12}(\T^d))$ fails to be $C^2$-differentiable at the origin.
\end{theorem}

In order to establish Theorem \ref{theorem:main}, we bring new insight from recent development in the Fourier restriction theory which is so called \textit{decoupling theory} (Wolff's inequality) into the study of the periodic Zakharov system.  We refer to books \cite{Stein,SoggeBook} for a general reference to the Fourier restriction theory and its connection to PDEs. 
Whilst there are several works that directly apply the consequence itself from the decoupling theory to the study of nonlinear PDE problems, as far as we are aware, there are only a few works that apply modern techniques or ideas in the proof of the decoupling theory. 
As an example, we refer to the work by Fan--Staffilani--Wang--Wilson \cite{FSWW} where they obtained global well-posedness for the cubic defocusing nonlinear Schr\"{o}dinger equation on irrational tori using the bilinear decoupling type estimate. 
Similarly, our broad scope in this paper is to tackle nonlinear PDE problems by introducing new techniques and ideas developed in Fourier restriction theory. 
At the same time, the study of the periodic Zakharov system raises a question on the decoupling theory - what if one has more than two surfaces interacting with each other which is reminiscent of the bilinear and multilinear Fourier restriction theory, \cite{Ben04,Ben05,BCT,BourGuth11,KM93,TVV}. See also the book \cite{Dem} and references therein. 

\subsection{Introduction to $\ell^2$-decoupling theory}
In order to make the paper accessible to readers from different backgrounds, we begin with a brief introduction to the $\ell^2$-decoupling theory.  We refer to the survey paper \cite{GuthSurvey} and the book \cite{Dem} for more details. 
Our starting point for the introduction of the decoupling inequality is the classical Littlewood--Paley theory. 
We denote the dyadic annulus by $A_{1} := \{ (\xi,\tau) \in \mathbb{R}^{d+1}: |(\xi,\tau)| \lesssim 1 \}$ and $A_N := \{ (\xi,\tau) \in \mathbb{R}^{d+1}: |(\xi,\tau)| \sim N  \}$ for each $N \in 2^{\mathbb{N}}$. Note that $\{ A_N \}_{N \in 2^{\mathbb{N}_0}}$  provides an almost pairwise disjoint decomposition of $\mathbb{R}^{d+1}$. 
Similarly, we can decompose the Fourier support of sufficiently nice function $f:\mathbb{R}^{d+1}\to \mathbb{C}$ dyadically as $f \sim \sum_N f_{A_N}$ where we define
$$
f_{K}(x,t)
:= 
\int_{\mathbb{R}^{d+1}} e^{2\pi i(x,t)\cdot (\xi,\tau)} \chi_{K}(\xi,\tau) \hat{f}(\xi,\tau)\, d\xi d\tau
$$ 
for $K \subset \mathbb{R}^{d+1}$ in general, $\chi_K$ is a smooth cut-off function on $K$, and $\hat{f}$ denotes the space-time Fourier transform. 
Then the Littlewood--Paley theory exploits an $L^2$-orthogonality of $\{f_{A_N} \}_{ N\in 2^{\mathbb{N}_0} }$ and ensures that\footnote{We will frequently use the notation $A\lesssim B$ to mean $A\le C B$ for some constant $C>0$.  When we emphasize the dependence of a parameter $a$ in the inequality, we write $A\lesssim_a B$ which means $A\le C_a B$ for some constant $C_{a}$ depending on $a$. } 
\begin{equation}\label{e:ClassLP}
\big\|
\big(
\sum_{N} 
|f_{A_N}|^2
\big)^\frac12
\big\|_{L^p(\mathbb{R}^{d+1})}
\lesssim
\| f \|_{ L^p(\mathbb{R}^{d+1}) }
\lesssim
\big\|
\big(
\sum_{N} 
|f_{A_N}|^2
\big)^\frac12
\big\|_{L^p(\mathbb{R}^{d+1})}
\end{equation}
for all $1<p<\infty$. 
At this stage, one may wonder what if the dyadic annuli $\{A_N\}_N$ are replaced by other decomposition. 
Such a question has been addressed by Carleson \cite{Carleson}, C\'{o}rdoba \cite{Cordo}, and Rubio de Francia \cite{Rubio}. 
For instance, Rubio de Francia \cite[Theorem B]{Rubio} proved that the inequality on the left-hand side of \eqref{e:ClassLP} remains true for decomposition into congruent cubes. 
More precisely, the following inequality
\begin{equation}\label{e:Rubio}
\big\| \big( \sum_j |f_{Q_j}|^2 \big)^\frac12 \big\|_{L^p(\mathbb{R}^{d+1})}
\lesssim 
\| f \|_{L^p(\mathbb{R}^{d+1})}
\end{equation}
holds true as long as $2\le p < \infty$. Here, $\{Q_j\}_j$ is a family of translated congruent cubes with arbitrary fixed side-length. 
On the other hand, the reverse inequality of \eqref{e:Rubio} failed dramatically except for the case $p=2$.  
Regardless of such failure in general, Fefferman \cite{Feff} found out that the curvature of the Fourier support plays a role in reversing \eqref{e:Rubio}.  
To be more precise, we introduce some notation which we will use throughout this paper. Denote a truncated paraboloid by 
$$
\mathbb{P}^d:= \{ ( \xi, |\xi|^2) \in \mathbb{R}^{d+1}: |\xi|\le 1 \},
$$
and its $N^{-2}$-neighborhood by 
$$
\mathcal{N}_{N^{-2}}( \mathbb{P}^d )
:= 
\Bigl\{
(\xi,\tau) \in \mathbb{R}^{d+1}:
|\xi|\le1,\; 
|\tau - |\xi|^2| \le \frac1{N^{2}}
\Bigr\}
$$
for $d\ge1$ and $N \gg 1$.  
We will decompose $\mathbb{P}^d$ or $\mathcal{N}_{N^{-2}}( \mathbb{P}^d )$ into small caps. For each $N\ge1$, let $\mathcal{C}_{N^{-1}}$ be a family of disjoint $\frac1N \times \cdots \times \frac1N \times \frac{1}{N^2}$ caps of the form 
\begin{equation}\label{e:DefCap}
\theta = 
\Bigl\{ 
( \xi,\tau) : \xi \in [- \frac1{2N}, \frac1{2N}]^d + c_\theta,\; |\tau - |\xi|^2 | \le \frac1{N^2} 
\Bigr\},
\end{equation}
where $c_\theta$ runs over $\frac1N\mathbb{Z}^d \cap [-1,1]^d$. 
With this notation, Fefferman's inequality \cite{Feff} states that if $f:\mathbb{R}^{d+1} \to \mathbb{C}$ satisfies ${\rm supp}\, \hat{f} \subset \mathcal{N}_{N^{-2}}(\mathbb{P}^d)$, then 
\begin{equation}\label{e:SquareConj}
\| f \|_{L^p(\mathbb{R}^{d+1})} \lesssim \big\| \big( \sum_{\theta \in \mathcal{C}_{N^{-1}}} |f_\theta|^2 \big)^\frac12 \big\|_{L^p(\mathbb{R}^{d+1})}
\end{equation}
as long as $2\le p \le 4$ and $d=1$. 
For $d\ge2$, \eqref{e:SquareConj} is conjectured to be true as long as $2\le p \le \frac{2(d+1)}{d}$. One can easily check \eqref{e:SquareConj} for $p=2$ by Plancherel's theorem, hence the difficulty is to prove \eqref{e:SquareConj} for large $p$.  
This problem is open for all $d\ge2$ and we refer to \cite{GWZ,Lee16,LV,SoggeBook} for the recent progress in this problem and relation to the Fourier restriction theory. 
The $\ell^2$-decoupling inequality we are concerned with in this paper is of the form\footnote{Precisely speaking \eqref{e:DecoupFat} means that for any $\varepsilon>0$ there exists some finite constant $C_\varepsilon$ such that 
$$
\| f \|_{L^p(\mathbb{R}^{d+1})} \le C_\varepsilon N^\varepsilon  \big( \sum_{\theta \in \mathcal{C}_{N^{-1}}} \|f_\theta\|_{L^p(\mathbb{R}^{d+1})}^2 \big)^\frac12 
$$
holds for all $N\ge1$ and $f$ such that ${\rm supp}\, \hat{f} \subset \mathcal{N}_{N^{-2}}(\mathbb{P}^d)$.} 
\begin{equation}\label{e:DecoupFat}
\| f \|_{L^p(\mathbb{R}^{d+1})} \lesssim_\varepsilon N^\varepsilon  \big( \sum_{\theta \in \mathcal{C}_{N^{-1}}} \|f_\theta\|_{L^p(\mathbb{R}^{d+1})}^2 \big)^\frac12 
\end{equation}
for any small $\varepsilon>0$ and any $f$ such that ${\rm supp}\, \hat{f} \subset \mathcal{N}_{N^{-2}}(\mathbb{P}^d)$. 
Thanks to Minkowski's inequality, \eqref{e:DecoupFat} is clearly weaker than \eqref{e:SquareConj} if $p \ge 2$ and hence one expects a wider range of $p$ for which \eqref{e:DecoupFat} is true. 
Inequalities of this type were initiated by Wolff in his work on the local smoothing conjecture for the wave equation  \cite{Wolff00} where he introduced $\ell^p$-decoupling inequality for the cone. 
So, inequalities of this type are also called \textit{Wolff's inequality}. 
Wolff's work was further investigated by \cite{Bour13,GarSee,GarSee10,LW02,LP06,PraSee07} and the almost sharp result was established by Bourgain--Demeter \cite{BourDem15}. 

\begin{theorem}[Theorem 1.1 in \cite{BourDem15}]\label{t:BD15}
Let $d\ge1$, $N \ge 1$ and $2\le p \le \frac{2(d+2)}d$. 
Then \eqref{e:DecoupFat} holds for any $\varepsilon>0$ if ${\rm supp}\, \hat{f} \subset \mathcal{N}_{N^{-2}}(\mathbb{P}^d)$.
Moreover, the range of $p$ is sharp in the sense that $\varepsilon$ in \eqref{e:DecoupFat} cannot be arbitrarily small if $p> \frac{2(d+2)}d$. 
\end{theorem}

Among several consequences of Theorem \ref{t:BD15} to broad fields of mathematics, we only mention the discrete restriction for paraboloid or the Strichartz estimate on torus which was initiated by Bourgain \cite{BourGAFA93}, see also works by Burq--G\'{e}rard--Tzvetkov \cite{BGT}, Guo--Oh--Wang \cite{GOW} and Vega \cite{Vega}. 
We refer the reader to the survey paper \cite{Pie} for other consequences to analytic number theory and a comprehensive introduction to the theory, see also \cite{GLYZ21,GuthSurvey,Li}. 

\begin{theorem}[Theorem 2.4 in \cite{BourDem15}]\label{theorem:StrichartzS-BD}
Let $d \geq 1$, $N \ge1$ and $2 \le p \le \frac{2(d+2)}d$. Then
\begin{equation}\label{e:StrichartzS}
\|e^{it \Delta} \phi\|_{L^p_{t,x}( \mathbb{T}^{d+1})} \lesssim_\varepsilon  N^{\varepsilon}\| \phi\|_{L_x^2(\mathbb{T}^d)}
\end{equation}
holds for all $\varepsilon>0$ if ${\rm supp}\, \hat{\phi} \subset \{ k \in \mathbb{Z}^d: |k| \le N \} $. 
Moreover, the range of $p$ is sharp in the sense that $\varepsilon$ in \eqref{e:StrichartzS} cannot be arbitrary small if $p> \frac{2(d+2)}d$. 
\end{theorem}
If one writes $a_k = \hat{\phi}(k)$, then \eqref{e:StrichartzS} can be read as 
$$
\Big\| \sum_{k \in \mathbb{Z}^d: |k|\le N} a_k e^{i (x\cdot k + t|k|^2)} \Big\|_{L^p_{t,x}(\mathbb{T}^{d+1})}
\lesssim_\varepsilon N^\varepsilon 
\| a_k \|_{\ell^2}
$$
which is reminiscent of the discrete variant of Stein--Tomas Fourier restriction estimates. 

\subsection{Trilinear discrete Fourier restriction estimate and improvement of the $\ell^2$-decoupling inequality under the shell constraint}
In the study of the periodic Zakharov system, it is important to exploit the resonance phenomenon between Schr\"{o}dinger and wave equations. 
Through the $X^{s,b}$ analysis,  this can be capitalized by the following trilinear estimate:
\begin{equation}\label{e:Trilinear}
\Big|
\int_{[- \pi, \pi] \times \mathbb{T}^d} e^{ it\Delta } \phi_1 \overline{e^{it\Delta}\phi_2} e^{\pm i t |\nabla|} \phi_3\, dtdx
\Big|
\lesssim 
N^\alpha 
\prod_{j=1}^3 \| \phi_j \|_{L^2(\mathbb{T}^d)},
\end{equation}
for ${\rm supp}\, \hat{\phi_j} \subset \{ |\xi| \sim N\} \cap \mathbb{Z}^d$ and some $\alpha \ge0$. 
For instance, Kishimoto \cite[Proposition 3.6]{Kishi13} obtained his result Theorem \ref{theorem:Kishimoto13} by proving\footnote{Precisely speaking, he proved more general trilinear estimates for functions whose Fourier supports are in some neighborhoods of hypersurfaces. See forthcoming Proposition \ref{proposition:KeyBilinearEst}.} \eqref{e:Trilinear} with $\alpha = \frac{d-2}2$ for $d\ge3$.
It is worth remarking that one needs to establish trilinear estimates of more general form where linear solutions are replaced by functions whose Fourier supports are contained in some neighborhoods of hypersurfaces in order to apply it to the study of \eqref{Zakharov}, see also Lemma \ref{lemma:TransferencePrinciple} below. Nevertheless, we introduced the estimate in this special form as  it might have its own interest from the viewpoint of the Fourier restriction theory. 
In \cite{Kishi13}, Kishimoto expected that the exponent $\alpha = \frac{d-2}2$ might not be sharp and it would be improved by involving some number theoretical argument. 
Hence, it is natural to see what can be said about \eqref{e:Trilinear} by appealing to the decoupling theory, especially Theorem \ref{theorem:StrichartzS-BD}. 
Let us demonstrate the case $d=3$. 
We can apply H\"{o}lder's inequality, \eqref{e:StrichartzS} with $p = \frac{10}{3}$ and the linear estimate for wave solutions (see \eqref{est:CorollaryWaveStrichartzT^d} below)
\[
\|e^{\pm it |\nabla |} \phi_3 \|_{L_{t}^{10}L_{x}^{\frac52} ([-\pi,\pi] \times \T^3)} \lesssim N^{\frac15} \|\phi_3\|_{L_{x}^2(\T^3)}
\] 
to see \eqref{e:Trilinear} with $ \alpha = \frac15 + \varepsilon$ as
\begin{align*}
&\big|
\int_{[-\pi,\pi] \times \mathbb{T}^{3}} e^{ it\Delta } \phi_1 \overline{e^{it\Delta}\phi_2} e^{\pm i t |\nabla|} \phi_3\, dtdx
\big|\\
&\le
\| e^{ it\Delta } \phi_1 \|_{L^{\frac{10}{3}}_{t,x}(\mathbb{T}^{3+1})} 
\| e^{ it\Delta } \phi_2 \|_{L^{\frac{10}{3}}_{t,x}(\mathbb{T}^{3+1})} 
\| e^{ \pm it|\nabla| } \phi_3 \|_{L_{t,x}^{\frac52} ([-\pi,\pi] \times \T^3)} \\
&\lesssim
\| e^{ it\Delta } \phi_1 \|_{L^{\frac{10}{3}}_{t,x}(\mathbb{T}^{3+1})} 
\| e^{ it\Delta } \phi_2 \|_{L^{\frac{10}{3}}_{t,x}(\mathbb{T}^{3+1})} 
\| e^{ \pm it|\nabla| } \phi_3 \|_{L_{t}^{10}L_{x}^{\frac52} ([-\pi,\pi] \times \T^3)} \\
&\lesssim 
N^{\frac15 + \varepsilon}
\prod_{j=1}^3 \| \phi_j \|_{L^2(\mathbb{T}^3)},
\end{align*}
which consequently, improves Kishimoto's result by virtue of Theorem \ref{theorem:StrichartzS-BD}\footnote{{An immediate consequence of \eqref{e:StrichartzS} is $\|e^{it \Delta} \phi\|_{L_t^q L_x^p(\T^{1+d})} \lesssim N^{\varepsilon}\|\phi\|_{L_x^2(\T^d)}$ where $\frac2q = d(\frac12 - \frac1p)$ and $2 \leq p \leq q$. 
If this estimate can be extended to the case $2 \leq q < p$ on the same lines as in the Euclidean case, we may show \eqref{e:Trilinear} with an optimal $\alpha$ up to $\varepsilon$ loss. However, as far as we are aware, the above estimate with $2 \leq q<p$ satisfying $\frac2q = d(\frac12 - \frac1p)$ remains an open problem. We refer the reader to Appendix~6.2.}
}. 
On the other hand, a simple example $\phi_j (\xi) = e^{i \xi_j^* \cdot x}$ where 
$$
\xi_1^* = (N,0,0),\;\;\;\xi_2^* = (-N-1, 0, 0),\;\;\; \xi_3^* = (-2 N-1,0,0), 
$$
shows that $\alpha \ge0$ is necessary for \eqref{e:Trilinear}. 
The aim of the paper is to fill this blank by exploiting an interaction between Schr\"{o}dinger and wave equation, or paraboloid and cone in the Fourier restriction language. 
Our first observation is that this interaction reduces the estimate \eqref{e:Trilinear} to the case when ${\rm supp}\, \hat{\phi_j}$, $j=1,2$, are contained in some shell. Namely we will see that it suffices to consider $\phi_1,\phi_2$ such that ${\rm supp}\, \hat{\phi_j} \subset \{ k \in \mathbb{Z}^d : c_j-1 \le |k| \le c_j + 1\}$ for some $c_j \sim N$, see the proof of Proposition \ref{proposition:KeyBilinearEst} for details of this reduction. We mention that Kishimoto \cite{Kishi13} also employed similar reduction to the Fourier support on some shell. 
By virtue of this reduction,  we are lead to the Strichartz estimate under the additional constraint of 
\begin{equation}\label{e:ShellStri}
{\rm supp}\, \hat{\phi} 
\subset 
\{ k \in \mathbb{Z}^d : c_*-1 \le |k| \le c_* + 1\}
\end{equation}
for some $c_* \sim N$, and expect some improvement of Theorem \ref{theorem:StrichartzS-BD}. 
Similarly, one can also ask some improvement of the decoupling inequality Theorem \ref{t:BD15} under the constraint 
\begin{equation}\label{e:ShellDecoup}
{\rm supp}\, \hat{f} 
\subset 
\Bigl\{
(\xi,\tau) \in \mathbb{R}^{d+1}:
d_*- \frac1N \le |\xi| \le d_* + \frac1N,\; 
|\tau - |\xi|^2| \le \frac1{N^{2}}
\Bigr\}
\end{equation}
for some $d_* \sim 1$. 
In fact, this turns out to be true and we give an almost sharp result.

\begin{theorem}\label{theorem:ShellDecoupling}
Let $d \geq 2$, $N \ge1$ and $2\le p \le \frac{2(d+1)}{d-1}$.
Then \eqref{e:DecoupFat} holds for all $\varepsilon>0$ and all $f$ satisfying  \eqref{e:ShellDecoup}. 
Moreover, the range of $p$ is sharp in the sense that $\varepsilon$ in \eqref{e:DecoupFat} cannot be arbitrarily small under the constraint \eqref{e:ShellDecoup} if $p > \frac{2(d+1)}{d-1}$. 
\end{theorem}
Recall that Bourgain--Demeter's range in Theorem \ref{t:BD15} was $2\le p\le \frac{2(d+2)}d$. 
Hence our contribution in Theorem \ref{theorem:ShellDecoupling} is to establish \eqref{e:DecoupFat} on the range $\frac{2(d+2)}d < p\le \frac{2(d+1)}{d-1}$ under the assumption \eqref{e:ShellDecoup} and the sharpness of the range.
It is nowadays standard to derive the Strichartz estimate from the decoupling inequality, see \cite{Bour13}. 
We state our Strichartz estimate on the shell in a slightly general form with the aid of application to \eqref{Zakharov}.  
\begin{theorem}\label{theorem:StrichartzShell}
Let $d \geq 2$, $N_1$, $N_2 \in 2^{\N_0}$ satisfy $N_2 \leq N_1$ and $2 \le p \le \frac{2(d+1)}{d-1}$. 
Then 
\begin{equation}\label{e:ShellStrichartz}
\|e^{it \Delta} \phi\|_{L^p_{t,x}( \mathbb{T}^{d+1})} \lesssim_\varepsilon  N_2^{\varepsilon}\| \phi\|_{L_x^2(\mathbb{T}^d)}
\end{equation}
holds for all $\varepsilon>0$ and all $\phi \in L^2(\mathbb{T}^d)$ satisfying 
\begin{equation}\label{e:AssumpStri}
{\rm supp}\, \hat{\phi} 
\subset 
\{ k \in \mathbb{Z}^d : c_*-1 \le |k| \le c_* + 1\} \cap B_{N_2}
\end{equation}
for some $c_* \sim N_1$ and $B_{N_2} \subset \R^d$, a ball of radius $N_2$ with arbitrary centre.  
\end{theorem}

From this shell Strichartz estimate, we can give an improvement for the trilinear discrete restriction estimate \eqref{e:Trilinear}. 

\begin{corollary}\label{Cor:Trilinear}
Let $d\ge3$ and $s_0$ be in \eqref{assumption:regularity}. Then \eqref{e:Trilinear} holds for $\alpha> s_0 - \frac12$.
\end{corollary}

\begin{remark}
\begin{enumerate}
    \item 
Whilst Corollary \ref{Cor:Trilinear} provides almost sharp estimate for \eqref{e:Trilinear} when $d=3$,  the problem of identifying the sharp $\alpha$ for \eqref{e:Trilinear} is still open for $d\ge4$. For instance if one could prove \eqref{e:Trilinear} with $\alpha = 0$ (or $\alpha = \varepsilon$) when $d=4$, then one might obtain the local well-posedness of \eqref{Zakharov} with optimal regularity $s = 0$ (or $s =\varepsilon$) as long as one utilizes the iteration argument.
    \item
    After we completed this paper, it was pointed out by Changkeun Oh that the $\ell^p$-decoupling inequality with an appropriate geometric constraint on ${\rm supp}\, \hat{f}$, was also considered in Guo--Zorin--Kranich \cite[Theorem 2.5]{GZ} whose proof was motivated from the argument by Oh \cite{Oh}. Hence, up to a difference between $\ell^p$-decoupling and $\ell^2$-decoupling, the initiation of the study of the decoupling inequality with an additional constraint on ${\rm supp}\, \hat{f}$ is attributed to Guo--Zorin--Kranich.  
    We will give further detailed comparison about this point after the proof of Theorem \ref{theorem:ShellDecoupling} in Subsection \ref{subsection3.3}.  
    It is interesting to point out that the work by Guo--Zorin--Kranich was motivated by the study of the decoupling inequality for the quadratic surface whose co-dimension is higher than one whilst Theorem \ref{theorem:ShellDecoupling} naturally emerges from a requirement of the study of the periodic Zakharov system. To us, it is not obvious if there is any link between these two subjects. 
\end{enumerate}

\end{remark}

We end this section by giving a comparison to other works.
The trilinear form in \eqref{e:Trilinear} comes from the bilinear estimate for $e^{it\Delta} \phi e^{\pm it|\nabla|} \psi$ with $X^{s,b}$ analysis via duality argument. 
In this regard, it would be of interest to seek a bilinear discrete restriction estimate of the form
$$
\| e^{it\Delta} \phi e^{\pm it|\nabla|} \psi \|_{L^p_{t,x}([-\pi,\pi]\times \mathbb{T}^{d})}
\lesssim  N^{\alpha} M^\beta \| \phi\|_{L^2(\mathbb{T}^d)} \| \psi \|_{L^2(\mathbb{T}^d)}
$$
for appropriate $p,\alpha,\beta$ where ${\rm supp}\, \hat{\phi} \subset \{ |k| \sim N \}$ and ${\rm supp}\, \hat{\psi} \subset \{ |k| \sim M \}$. 
In the continuous case, such a problem is addressed by Candy \cite{Candy21}. See also \cite{Candy19}.

\vspace{2mm}
The rest of the paper is organized as follows: in Section \ref{section:Preliminaries}, we introduce the $X^{s,b}$-type spaces and their properties. In Section \ref{section:DecouplingTheory}, we prove the $\ell^2$-decoupling inequality for functions under shell constraint (Theorem \ref{theorem:ShellDecoupling}) which readily yields the shell Strichartz estimate (Theorem \ref{theorem:StrichartzShell}). 
In Section \ref{section:ProofMainTheorem}, by employing the Strichartz estimate obtained in Section \ref{section:DecouplingTheory}, we handle the nonlinear interactions of Schr\"{o}dinger and wave equations and establish Theorem \ref{theorem:main}. Section \ref{section:ProofNegativeResult} is devoted to the proof of Theorem \ref{theorem:NotC^2}. Lastly in the Appendix, as a consequence of Theorem \ref{theorem:StrichartzShell} for $d=3$ and $p=4$, we introduce a bound of integer solutions of a Diophantine system under the shell constraint.

\section{Preliminaries}\label{section:Preliminaries}
In this section, we introduce the function spaces and collect some fundamental estimates.

Let us first rewrite the original Zakharov system as a first-order system, which is a standard reduction in the study of the Zakharov system. See \cite{GTV97}, \cite{BHHT09}, \cite{Kishi13}. 
Let $w = n + i \langle \nabla \rangle^{-1} \partial_t n$ and $w_0 = n_0 + i \langle \nabla \rangle^{-1} n_1$. 
Then, we may rewrite the system \eqref{Zakharov} as
\begin{equation}\label{Zakharov2}
\begin{cases}
i \partial_t u  + \Delta u = \frac12 (w + \overline{w})u,\\
i \partial_t w - \langle \nabla \rangle w = - \langle \nabla \rangle^{-1} \Delta (|u|^2) - \langle \nabla \rangle^{-1}\bigl( \frac{w + \overline{w}}{2} \bigr),\\
(u(0), w(0)) = (u_0,w_0) \in H^s(\T^d) \times H^{\ell}(\T^d).
\end{cases}
\end{equation}
It is easily seen that the local well-posedness of \eqref{Zakharov2} in $H^s(\T^d) \times H^{\ell}(\T^d)$ implies that of \eqref{Zakharov} in $\mathcal{H}^{s,\ell}(\T^d)$. 
Thus, we consider \eqref{Zakharov2} instead of \eqref{Zakharov} hereafter.

\begin{definition}
Let $\eta :\R \to \R$ be a smooth function satisfying $\eta=1$ on $[-1,1]$ and $\supp \eta \subset (-2,2)$. Let $N \in 2^{\N_0}$ with $\N_0 = \N \cup \{0\}$. Define
\[
\eta_1 = \eta, \quad \eta_{N}(r) = \eta \Bigl( \frac{r}{N} \Bigr) - \eta \Bigl( \frac{2r}{N}\Bigr) \quad (N \geq 2).
\]
$\{P_N\}_{N \in 2^{\N_0}}$ denotes the collection of standard Littlewood--Paley operators defined by $P_N = \F_{k}^{-1} \eta_{N}(|k|) \F_x$. Let $\widetilde{u} (\tau,k)= \F_{t,x} u (\tau,k)$, $L \in 2^{\N_0}$ and
\[
Q_L^S u =  \F_{t,x}^{-1} \big( \eta_{L}(\tau + |k|^2) \widetilde{u} \big), \qquad Q_L^{W_{\pm}} u = \F_{t,x}^{-1}  \big( \eta_{L}(\tau\pm \langle k \rangle) \widetilde{u}\big).
\]
We write $P_{N,L}^S = P_{N} Q_L^S$ and $P_{N,L}^{W_{\pm}} = P_N Q_{L}^{W_{\pm}}$.

We define the function spaces $X_S^{s,b}$ and $X_{W_{\pm}}^{s,b}$ as follows:
\begin{align*}
& X_S^{s,b} = \bigl\{ u \in \mathcal{S}'(\R \times \T^d) \, : \,  \|u\|_{X_S^{s,b}} = \bigl( \sum_{N,L} L^{2b} N^{2s} \| P_{N,L}^S u \|_{L_{t,x}^2}^2 \bigr)^{\frac12} < \infty\bigr\},\\
& X_{W_{\pm}}^{s,b} = \bigl\{ u \in \mathcal{S}'(\R \times \T^d) \, : \,  \|u\|_{X_{W_{\pm}}^{s,b}} = \bigl( \sum_{N,L} L^{2b} N^{2s} \| P_{N,L}^{W_{\pm}} u \|_{L_{t,x}^2}^2 \bigr)^{\frac12} < \infty \bigr\}.
\end{align*}
Let $T>0$ and $X$ be either $X_S^{s,b}$ or $X_{W_{\pm}}^{s,b}$. We define the time restricted space $X(T)$ as follows:
\begin{align*}
X(T) & = \{ u \in C([0,T);H^s(\T^d)) \, : \,  \|u\|_{X(T)} < \infty \},\\
\|u\|_{X(T)} & = \inf  \{\|U\|_{X} \, : \, U \in X, \ \  u(t) = U(t) \ \forall  t \in (0,T) \}.
\end{align*}
Notice that, in the case $b > \frac12$, the Sobolev embedding in time implies $\|u\|_{L^{\infty}_t H^s} \lesssim \|u\|_{X^{s,b}_S}$. Thus, if $b >\frac12$, $X^{s,b}_S(T)$ is a Banach space. The same holds for $X^{s,b}_{W_{\pm}}(T)$. 
\end{definition}
We introduce the well-known property of $X^{s,b}$-type spaces.
For the sake of completeness, we include the proof. 
\begin{lemma}[Transference principle]\label{lemma:TransferencePrinciple}
Let $U(t) \in \{e^{it\Delta}, e^{\mp it \langle \nabla \rangle}\}$. 
We use the notations
\[
X_{U}^{s,b}= \begin{cases} 
X_S^{s,b} & \mathrm{if} \ \ U(t) = e^{it\Delta},\\
X_{W_{\pm}}^{s,b} & \mathrm{if} \ \ U(t) = e^{\mp it \langle \nabla \rangle},
\end{cases}
\qquad Q_L^{U}  = \begin{cases} 
Q_L^{S}& \mathrm{if} \ \ U(t) = e^{it\Delta},\\
Q_L^{W_{\pm}} & \mathrm{if} \ \ U(t) = e^{\mp it \langle \nabla \rangle}.
\end{cases}
\]
Let $N \in 2^{\N_0}$ and $\phi \in L^2(\T^d)$. 
Assume that there exist $q$, $r \in [2,\infty]$, and $\alpha \in \R$ such that the following linear estimate holds.
\begin{equation}\label{ass:lemma1.4-Strichartz}
\|U(t) P_N \phi \|_{L_{[-1,1]}^q L_x^r} \lesssim N^{\alpha} \|P_N \phi\|_{L_x^2}.
\end{equation}
Then, for $L \in  2^{\N_0}$, we have
\begin{equation}\label{est:lemma1.4-goal}
\|  Q_{L}^{U} P_N u \|_{L_{t}^q L_x^r} \lesssim  L^{\frac12} N^{\alpha} \|Q_{L}^{U} P_N u\|_{L_{t,x}^2},
\end{equation}
where $L^q_t$ denotes $L^q_t(\mathbb{R})$.
\end{lemma}
\begin{proof}
First, we consider the case $q < \infty$. 
Let us choose $\psi \in \mathcal{S}(\R)$ so that $\supp \F_t \psi \subset (-2,2)$ and 
\begin{equation}\label{e:AssumpPsi}
2^{-10}
\le 
\big(
\sum_{m \in \mathbb{Z}} 
|\psi(t-m)|^{2q} 
\big)^\frac1q,
\quad 
\sum_{m \in \mathbb{Z}} 
|\psi(t-m)|
\le 2^{10}
\end{equation}
for any $t \in \R$. 
We define $\psi_m(\cdot) = \psi(\cdot - m)$. 
Since $\psi \in \mathcal{S}(\R)$ and $U(t)$ is unitary in $L^2(\T^d)$, it is easy to see that the estimate~\eqref{ass:lemma1.4-Strichartz} provides
\[
\|\psi(t)U(t) P_N \phi \|_{L_t^q L_x^r} \lesssim N^{\alpha} \|P_N \phi\|_{L_x^2}.
\]
Using the translation invariance of the Lebesgue measure on $\R$, we have 
\begin{equation}\label{e:TransEst}
\|\psi(t - m)U(t - m) P_N \phi \|_{L_t^q L_x^r} \lesssim N^{\alpha} \|P_N \phi\|_{L_x^2}
\end{equation}
for all $m \in \mathbb{Z}$. 
Our goal is to prove
\begin{equation}\label{est:lemma1.4-01}
\|  \psi_m^2 Q_{L}^{U} P_N u \|_{L_t^q L_x^r} \lesssim  L^{\frac12}  N^{\alpha} \| \psi_m Q_{L}^{U} P_N u\|_{L_{t,x}^2},
\end{equation}
where the implicit constant does not depend on $m \in \Z$. Let us see that this inequality implies~\eqref{est:lemma1.4-goal}. 
By applying the first inequality in \eqref{e:AssumpPsi}, the Minkowski inequality for $\ell^{\frac{q}2}_m$ in view of $q\ge2$, and the second  inequality in \eqref{e:AssumpPsi}, we obtain that 
\begin{align*}
\|  Q_{L}^{U} P_N u \|_{L_t^q L_x^r} 
&\lesssim 
\bigg\{ 
\int 
\bigg[
\big(
\sum_{m \in \mathbb{Z}} 
|\psi(t-m)|^{2q} 
\big)^\frac1q 
\|  Q_{L}^{U} P_N u (t) \|_{L_x^r}
\bigg]^q\, dt
\bigg\}^\frac1q\\
& =
\Bigl(  \int \sum_{m \in \Z} |\psi_m(t)|^{2q} \|Q_{L}^{U} P_N u \|_{L_x^r}^q dt \Bigr)^{\frac{1}{q}}\\
& \sim \Bigl( \sum_{m\in \Z}\| \psi_m^2 Q_{L}^{U} P_N u \|_{L_t^q L_x^r}^q \Bigr)^{\frac{1}{q}}\\
& \lesssim \Bigl( \sum_{m\in \Z} \big(L^{\frac12}  N^{\alpha} \big)^{q}
\| \psi_m Q_{L}^{U} P_N u\|_{L_{t,x}^2}^q\Bigr)^{\frac{1}{q}}\\
& \lesssim L^{\frac12} N^{\alpha} \|Q_{L}^{U} P_N u\|_{L_{t,x}^2}.
\end{align*}

We consider~\eqref{est:lemma1.4-01}. It is straightforward to check that
\begin{align*}
& \F_{t,x} \bigl( e^{-it \Delta} Q_L^{S}u\bigr) (\tau,k) = \eta_L(\tau)\widetilde{u}(\tau-|k|^2,k),\\
& \F_{t,x}\bigl(e^{\pm it \langle \nabla \rangle}Q_L^{W_{\pm}} u \bigr)(\tau,k) = \eta_L(\tau) \widetilde{u}(\tau \mp \langle k \rangle, k).
\end{align*}
Therefore, it follows from $\supp \F_t \psi \subset (-2,2)$ and $L \in 2^{\N_0}$ that
\begin{equation}\label{est:lemma1.4-02}
\F_{t,x} (U(-t) \psi_m Q_L^U P_N u) (\tau,k) = \1_{4L}(\tau) \F_{t,x} (U(-t) \psi_m Q_L^U P_N u) (\tau,k)
\end{equation}
Here $\1_{L}$ denotes the characteristic function of the set $\{|\tau| \leq L\}$.
From \eqref{est:lemma1.4-02}, we obtain
\[
u(t) = \int e^{it \tau} U(t)\bigl( \F_t U(- \cdot) u \bigr)(\tau) d \tau,
\]
and from \eqref{e:TransEst} that
\begin{align*}
\| \psi_m^2 Q_{L}^{U} P_N u \|_{L_t^q L_x^r} & = \Bigl\|  \int e^{it \tau} \psi_m (t)U(t)\bigl( \F_t U(- \cdot)\psi_mQ_{L}^{U} P_N u\bigr) d \tau \Bigr\|_{L_t^q L_x^r}\\
& \lesssim  \int \Bigl\|  \psi(t-m)U(t-m)U(m)\bigl( \F_t U(- \cdot)\psi_m Q_{L}^{U} P_N u\bigr)  \Bigr\|_{L_t^q L_x^r} d \tau\\
& \lesssim  N^{\alpha} \int \|\F_{t,x}U(- \cdot)\psi_m Q_{L}^{U} P_N u\|_{L_{k}^2}d \tau\\
& =   N^{\alpha}\int \1_{4L}(\tau) \| \F_{t,x}U(- \cdot)\psi_m Q_{L}^{U} P_N u\|_{L_{k}^2}d \tau \\
& \lesssim 
 L^{\frac12}  N^{\alpha} \|\psi_m Q_{L}^{U} P_N u\|_{L_{t,x}^2}.
\end{align*}
This completes the proof for the case $q< \infty$.

We assume $q= \infty$. For any $T \in \R$, it follows that
\[
\|\psi(t-T)U(t)P_N \phi\|_{L_t^{q}L_x^r} = \|\psi(t)U(t) U(T) P_N \phi\|_{L_t^q L_x^r} \lesssim N^{\alpha} \|P_N\phi\|_{L_x^2}.
\]
Hence, when $q= \infty$, we have $\|U(t) P_N \phi\|_{L_t^{\infty}L_x^r} \lesssim  N^{\alpha} \|P_N \phi\|_{L_x^2}$. In the same way as above, we can verify the claim.
\end{proof}

\begin{remark}\label{remark:TransferencePrinciple}
Let $U(t) \in \{e^{it\Delta}, e^{\mp it \langle \nabla \rangle} \}$. Define
\begin{equation*}
    \Phi_U(k) :=
\begin{cases}
|k|^2 \qquad & \mathrm{if} \ U(t) = e^{it\Delta},\\
\pm \langle k \rangle & \mathrm{if} \ U(t) = e^{\mp it\langle \nabla \rangle}.
\end{cases}
\end{equation*}
For $\tau_0 \in \R$, notice that $\F_{t,x}Q_L^U (e^{it \tau_0}u) = \eta_{L} (\tau - \tau_0 + \Phi_U(k) ) \widetilde{u}$. Hence, Lemma~\ref{lemma:TransferencePrinciple} implies that if \eqref{ass:lemma1.4-Strichartz} holds then for any $\tau_0 \in \R$ it follows that
\[
\| u_{N,L,\tau_0} \|_{L_{t}^q L_x^r} \lesssim L^{\frac12} N^{\alpha} \| u_{N,L,\tau_0}\|_{L_{t,x}^2}, 
\]
where $u_{N,L,\tau_0}$ satisfies
\begin{equation*}
  \supp~\widetilde{u}_{N,L,\tau_0} \subset \{(\tau,k) \in \R \times \Z^d \, : \, |\tau-\tau_0 -\Phi_{U}(k)| \lesssim L, \quad 
|k| \sim N \}.  
\end{equation*}
\end{remark}

Combining Theorem \ref{theorem:StrichartzS-BD} and Lemma \ref{lemma:TransferencePrinciple}, we have the following corollary.
\begin{corollary}\label{corollary:StrichartzS}
Let $N$, $L \in 2^{\N_0}$. Assume that $q, \, p \in [2, \infty]$ satisfy
\[
\frac{2(d+2)}{d} \leq q \leq \infty, \qquad \frac{1}{q} = \frac{d}{2} \Bigl(\frac12 - \frac1p \Bigr).
\]
Then, we have
\begin{equation}\label{est:StrichartzS}
\| P_{N,L}^S u\|_{L_t^q L_x^p} \lesssim_\varepsilon L^{\frac12}N^{\varepsilon}\|P_{N,L}^S u\|_{L_{t,x}^2}
\end{equation}
for all $\varepsilon>0$. 
\end{corollary}
\begin{proof}
It follows from Lemma~\ref{lemma:TransferencePrinciple} and Theorem~\ref{theorem:StrichartzS-BD} that
\begin{equation}\label{est:StrichartzS-pure}
\| P_{N,L}^S u\|_{L_{t,x}^{\frac{2(d+2)}{d}}} \lesssim L^{\frac12}N^{\varepsilon}\|P_{N,L}^S u\|_{L_{t,x}^2}.
\end{equation}
By interpolating \eqref{est:StrichartzS-pure} and the trivial estimate 
$\| P_{N,L}^S u\|_{L_t^{\infty} L_x^2} \lesssim L^{\frac12}\|P_{N,L}^S u\|_{L_{t,x}^2}$, we obtain \eqref{est:StrichartzS}.
\end{proof}
Next we recall the Strichartz estimate for a solution of the wave equation on $\R^{d+1}$. 
For the details, we refer to \cite{GV95}, \cite{KT98}.
\begin{theorem}\label{theorem:WaveStrichartzR^d}
Let $d \geq 2$ and assume that $q, \, p \in [2, \infty]$ satisfy
\begin{equation}\label{ass:WaveAdmissible}
\frac1q = \frac{d-1}{2} \Bigl(\frac12 - \frac1p \Bigr), \quad  (q,p,d) \not=(2,\infty,3).
\end{equation}
Then, we have

\begin{equation}\label{est:WaveStrichartzR^d}
\|  \mathcal{W} (t)( f, g) \|_{L_{t}^q L_{x}^p(\R \times \R^d)} \\
\lesssim  \| f\|_{\dot{H}^{\frac{d}{2}- \frac{d}{p}-\frac1q}(\R^d)} +\|  g\|_{\dot{H}^{\frac{d}{2}- \frac{d}{p} - \frac1q -1}(\R^d)},
\end{equation}
where 
\[
\mathcal{W} (t) (f,g) := \cos(t |\nabla|) f + \frac{\sin(t|\nabla|)}{|\nabla|} g.
\]

\end{theorem}
It is well-known that the solution to the linear wave equation possesses the finite speed of propagation property. We refer to \cite{Evans10}, \cite{Tao06}, \cite{Tzve19} for the details. 
By exploiting such a property, 
one can derive the Strichartz estimates for the linear wave equation under the periodic setting from those in the Euclidean space. 
See e.g. \cite{Tzve19}. 
Combining Theorem \ref{theorem:WaveStrichartzR^d} and Lemma \ref{lemma:TransferencePrinciple}, we prove the following:
\begin{corollary}\label{corollary:StrichartzW}
Let $d \geq 2$, $L, N \in 2^{\N_0}$ and assume that $q, \, p \in [2, \infty]$ satisfy \eqref{ass:WaveAdmissible}. Then, we have
\begin{equation}\label{est:StrichartzW}
\|w_{\pm}\|_{L_t^q L_x^p} \lesssim L^{\frac12} N^{\frac{d}{2}-\frac{d}{p} - \frac{1}{q}} \|w_{\pm}\|_{L_{t,x}^2},
\end{equation}
for $w_{\pm} \in L^2(\R \times \T^d)$ such that
\[
\supp \widetilde{w}_{\pm} \subset \{ (\tau,k) \in \R \times \Z^d \, : \, 
|\tau- \tau_0 \pm \langle k \rangle | \lesssim L, \ \  |k| \sim N \},
\]
where $\tau_0 \in \R$.
\end{corollary}
\begin{proof}
We give a proof for the sake of completeness although this property is folklore. 
Lemma~\ref{lemma:TransferencePrinciple} and \textit{Remark} \ref{remark:TransferencePrinciple} imply that it suffices to show

\begin{equation}
\label{est:CorollaryKGStrichartzT^d}
\|e^{\pm it  \langle \nabla \rangle} P_N \phi \|_{L_{t}^qL_{x}^p([-1,1] \times \T^d)} \lesssim N^{\frac{d}{2}- \frac{d}{p}-\frac1q} \|P_N \phi\|_{L_{x}^2(\T^d)}.
\end{equation}
We omit the proof of the trivial case $N=1$ and assume $N \geq 2$ hereafter. To see \eqref{est:CorollaryKGStrichartzT^d}, it is enough to prove
\begin{equation}
\label{est:CorollaryWaveStrichartzT^d}
\|e^{\pm it |\nabla |} P_N \phi \|_{L_{t}^qL_{x}^p([-1,1] \times \T^d)} \lesssim N^{\frac{d}{2}- \frac{d}{p}-\frac1q} \|P_N \phi\|_{L_{x}^2(\T^d)},
\end{equation}
where the propagator $e^{\pm it  \langle \nabla \rangle}$ in \eqref{est:CorollaryKGStrichartzT^d} is replaced with $e^{\pm it |\nabla |}$. 
Indeed, the Sobolev embedding and Plancherel's theorem yield
\begin{align*}
    & \|(e^{\pm it  \langle \nabla \rangle} - e^{\pm it |\nabla |} ) P_N \phi \|_{L_{t}^qL_{x}^p([-1,1] \times \T^d)} \\
    & \lesssim N^{\frac{d}{2} - \frac{d}{p}}\|(e^{\pm it  \langle \nabla \rangle} - e^{\pm it |\nabla |} ) P_N \phi \|_{L_{t}^{\infty}L_{x}^2([-1,1] \times \T^d)}\\
    & = N^{\frac{d}{2} - \frac{d}{p}} \sup_{t\in[-1,1]} \biggl( \sum_{k \in \Z^d} |(e^{\pm it \langle k \rangle} - e^{\pm it |k|}) \widehat{P_N \phi}(k)|^2 \biggr)^{\frac12}\\
    & \lesssim  N^{\frac{d}{2} - \frac{d}{p}-1} \|P_N \phi\|_{L_{x}^2(\T^d)}.
\end{align*}
Here we used the fact $| e^{\pm it \langle k \rangle} - e^{\pm it |k|}| \lesssim N^{-1}$ for all $(t,k) \in \R \times \Z^d$ such that $t \in [-1,1]$ and $|k| \sim N$.

We turn to the proof of \eqref{est:CorollaryWaveStrichartzT^d}. 
Because $N \geq 2$ and $e^{\pm it |\nabla|} = \cos (t |\nabla|) \pm i \sin (t|\nabla|) $, it suffices to show
\begin{equation}\label{est:WaveStrichartzT^d}
\begin{split}
\| & \mathcal{W} (t)(P_N f, P_N g) \|_{L_{t}^q L_{x}^p([-1,1] \times \T^d)} \\
& \lesssim N^{\frac{d}{2}- \frac{d}{p}-\frac1q} \|P_N f\|_{L_{x}^2(\T^d)} + N^{\frac{d}{2}- \frac{d}{p} - \frac1q -1}\| P_N g\|_{L_x^2(\T^d)}.
\end{split}
\end{equation}
Let $\psi \in \mathcal{S}(\R^d)$ satisfy $\psi \equiv 1$ on $[-10\pi, 10 \pi]^d$ and $\supp \psi \subset [- 20 \pi , 20 \pi]^d$. 
We define $(F_N,G_N) \in \mathcal{S}(\R^d) \times \mathcal{S}(\R^d)$ as $(F_N,G_N) = (\psi P_N f, \psi P_N g)$. 
Then, $\|F_N\|_{L^2(\R^d)} \sim \| P_N f \|_{L^2(\T^d)}$ and $\|G_N\|_{L^2(\R^d)} \sim \| P_N g \|_{L^2(\T^d)}$. 
In addition, because of the finite speed of propagation for a solution of the linear wave equation, 
for any $t \in [-1,1]$, it holds that $\mathcal{W}(t)(P_N f(x), P_N g(x)) = \mathcal{W}(t)(F_N(x), G_N(x))$ if $x \in [-\pi,\pi]^d$. Consequently, by using \eqref{est:WaveStrichartzR^d}, we obtain
\begin{align*}
&\|\mathcal{W}(t) (P_N f(x), P_N g(x))\|_{L_{t}^qL_{x}^p([-1,1] \times \T^d)}\\
& 
= \|\mathcal{W}(t) (F_N(x), G_N(x)) \|_{L_{t}^qL_{x}^p([-1,1] \times [-\pi, \pi]^d)}\\
& \lesssim  \| F_N \|_{\dot{H}^{\frac{d}{2}- \frac{d}{p}-\frac1q}(\R^d)} +\|  G_N\|_{\dot{H}^{\frac{d}{2}- \frac{d}{p} - \frac1q -1}(\R^d)}.
\end{align*}
Therefore, to see \eqref{est:WaveStrichartzT^d}, it suffices to prove
\begin{align}
\label{est:RevisionCorollary2.5-1}
& \| F_N \|_{\dot{H}^{\frac{d}{2}- \frac{d}{p}-\frac1q}(\R^d)} \lesssim N^{\frac{d}{2}- \frac{d}{p}-\frac1q} \|P_N f\|_{L^2(\T^d)},\\
\label{est:RevisionCorollary2.5-2}
&\|  G_N\|_{\dot{H}^{\frac{d}{2}- \frac{d}{p} - \frac1q -1}(\R^d)} \lesssim N^{\frac{d}{2}- \frac{d}{p}-\frac1q-1} \|P_N g\|_{L^2(\T^d)}.
\end{align}
Since $\frac{d}{2} - \frac{d}{p} - \frac{1}{q} \geq 0$, \eqref{est:RevisionCorollary2.5-1} follows from
\begin{equation}
\label{est:corollary1.9-1}
\sum_{M \in 2^{\N_0}} M^{\frac{d}{2} -\frac{d}{p} - \frac{1}{q}} \|P_{M} F_N \|_{L^2(\R^d)}  \lesssim 
N^{\frac{d}{2}- \frac{d}{p}-\frac1q} \|P_N f\|_{L^2(\T^d)}.
\end{equation}
Since $\|F_N\|_{L^2(\R^d)} \sim \| P_N f \|_{L^2(\T^d)}$, we only need to treat the cases $M \ll N$ and $M \gg N$. 
For $k \in\Z^d$, let $\widehat{\phi}(k)$ be the $k$-th Fourier coefficient of $\phi$. We may write
\[
\F_{x} F_N (\xi)= \sum_{k \in \Z^d} \widehat{\psi}(\xi-k) \eta_N(|k|) \widehat{f}(k).
\]
We note that $|\widehat{\psi}(\xi-k)| \lesssim \langle \xi-k \rangle^{- J}$ for any $J >0$. Then, it is observed that
\begin{align*}
& \sum_{M \ll N} M^{\frac{d}{2} -\frac{d}{p} - \frac{1}{q}} \|P_{M} F_N \|_{L^2(\R^d)}\\
& = \sum_{M \ll N} M^{\frac{d}{2} -\frac{d}{p} - \frac{1}{q}} \Bigl\| \eta_{M} \sum_{k \in \Z^d} \widehat{\psi}(\cdot-k) \eta_N(|k|) \widehat{f}(k) \Bigr\|_{L^2(\R^d)}\\
& \lesssim \sum_{M \ll N} M^{\frac{d}{2} -\frac{d}{p} - \frac{1}{q}} N^{- J}\|\eta_{M}\|_{L^2(\R^d)}  \sum_{k \in \Z^d}  |\eta_N(|k|)|  |\widehat{f}(k)|  \\
& \lesssim \sum_{M \ll N}M^{d -\frac{d}{p} - \frac{1}{q}} N^{-J+ \frac12 d} \|P_N f\|_{L^2(\T^d)}\\
&\lesssim \|P_N f\|_{L^2(\T^d)}.
\end{align*}
This completes the proof of \eqref{est:corollary1.9-1} in the case $M \ll N$. 
The case $M \gg N$ can be handled in a similar way.

Lastly, we prove \eqref{est:RevisionCorollary2.5-2}. Notice that $\frac{d}{2} - \frac{d}{p} - \frac{1}{q} -1 \leq 0$. It suffices to show
\[
\begin{split}
 \|P_1 G_N\|_{\dot{H}^{\frac{d}{2} - \frac{d}{p} - \frac{1}{q} -1}} & + \sum_{M \in 2^{\N}}  M^{\frac{d}{2} -\frac{d}{p} - \frac{1}{q}-1} \|P_{M} G_N \|_{L^2(\R^d)}\\
 & \lesssim 
N^{\frac{d}{2}- \frac{d}{p}-\frac1q-1} \|P_N g\|_{L^2(\T^d)}.
\end{split}
\]
The latter term can be handled in the same way as for \eqref{est:corollary1.9-1}. For the former term, the Sobolev embedding gives
\[
\|P_1 G_N\|_{\dot{H}^{\frac{d}{2} - \frac{d}{p} - \frac{1}{q} -1}(\R^d)} 
\lesssim \| P_1 G_N\|_{L^{\alpha}(\R^d)},
\]
where $\frac12 \leq \frac{1}{\alpha} := \frac{1}{p} + \frac{1}{qd} + \frac{1}{d} <1$\footnote{Strictly speaking, this embedding does not hold if $(d,q,p)=(2,\infty,2)$. In this case, however, the claim \eqref{est:StrichartzW} follows from the $L^2$ conservation and Lemma~\ref{lemma:TransferencePrinciple}.}. 
In the same way as above, for any $J >0$, it holds that
\begin{equation}
\label{est:RevisionCorollary2.5-3}
 \| P_1 G_N\|_{L^{2}(\R^d)} \lesssim N^{-J} \|P_N g\|_{L^2(\T^d)}.
\end{equation}
In addition, it is easy to see that
\begin{equation}
\label{est:RevisionCorollary2.5-4}
\| P_1 G_N\|_{L^{1}(\R^d)} \lesssim \| \psi P_N g\|_{L^1(\R^d)} \lesssim \|P_N g\|_{L^2(\T^d)}.
\end{equation}
It follows from \eqref{est:RevisionCorollary2.5-3} and \eqref{est:RevisionCorollary2.5-4} that 
\[
\| P_1 G_N\|_{L^{\alpha}(\R^d)} \lesssim N^{\frac{d}{2}- \frac{d}{p}-\frac1q-1} \|P_N g\|_{L^2(\T^d)}.
\]
This completes the proof of \eqref{est:RevisionCorollary2.5-2}.
\end{proof}

We can also derive fattened version of Theorem \ref{theorem:StrichartzShell} by employing Lemma \ref{lemma:TransferencePrinciple} (see also \textit{Remark} \ref{remark:TransferencePrinciple}), we have the following corollary. 
\begin{corollary}\label{corollary:StrichartzShell}
Let $d \geq 2$, $L \in 2^{\N_0}$, $N_1$, $N_2 \in 2^{\N_0}$ satisfy $N_2 \leq N_1$, and $p =\frac{2(d+1)}{d-1}$. 
Then we have
\begin{equation}\label{est:StrichartzSS}
\|u\|_{L_{t,x}^p} \lesssim_\varepsilon L^{\frac12}N_2^{\varepsilon}\|u\|_{L_{t,x}^2}
\end{equation}
for any $\varepsilon>0$ and $u \in L^2(\R \times \T^d)$ such that 
\[
\supp \widetilde{u} \subset \{(\tau,k) \in \R \times \Z^d \, : \,  |\tau- \tau_0-|k|^2| \lesssim L, \ c_* - 1  \leq |k| \leq c_* + 1, \ k \in B_{N_2}\},
\]
where $\tau_0 \in \R$, $c_* \sim N_1$ and $B_{N_2} \subset \R^d$ is a ball of radius $N_2$ with arbitrary centre. 
\end{corollary}

\section{Decoupling theory on the shell: Proofs of Theorem \ref{theorem:ShellDecoupling} and  Theorem \ref{theorem:StrichartzShell}}\label{section:DecouplingTheory}
\subsection{Decoupling inequality for the Fourier extension operator}
At this stage, it is worth to introduce the Fourier extension operator defined by 
$$
E_{\mathbb{P}^d} g(x,t) 
:= 
\int_{\mathbb{R}^d} 
e^{i( x\cdot \xi + t|\xi|^2)} g(\xi)\, d\xi,
\;\;\;
(x,t) \in \mathbb{R}^{d+1}
$$
for ${\rm supp}\, g \subset \{ \xi\in \mathbb{R}^d: |\xi|\le 2 \}$. 
In the forthcoming proof, we will use both $E_{\mathbb{P}^d} $ and $E_{\mathbb{P}^{d-1}}$. 
As one might expect, there exists an analogue of the decoupling inequality for $E_{\mathbb{P}^d}$. 
To state it, we introduce the notion of caps in this language. Let $\mathcal{C}_{N^{-1}}$ be a family of disjoint congruent cubes in $\mathbb{R}^d$ of the form 
$$
\theta = \Bigl[ - \frac1{2N}, \frac1{2N} \Bigr]^d + c_\theta
$$
where $c_\theta$ runs over $\frac1N \mathbb{Z}^d \cap [-1,1]^d$. 
This is slightly abusing notation compared to \eqref{e:DefCap} but we remark that if we write $\Phi(\xi) = (\xi, |\xi|^2)$, then $N^{-2}$-neighborhood of $\Phi(\theta)$ corresponds to the $\frac1N\times\cdots \times \frac1N\times \frac1{N^2}$ cap  introduced in \eqref{e:DefCap}. 
We also need to introduce a nice weight function which is nowadays a  common technique in order to justify the localisation argument. 
For a convex body $K \subset \mathbb{R}^{d+1}$, we use a weight function $w_K: \mathbb{R}^{d+1} \to \mathbb{R}_+$ that satisfies the following properties: (i) $\1_K \lesssim w_K$ and $w_K$ decays rapidly\footnote{We often choose $K$ as a ball in $\mathbb{R}^{d+1}$ with large radius $R$ and arbitrary centre $c$ in which case this condition means $w_B(t,x) \lesssim (1+| (x-c)/R |)^{-M} $ for sufficiently large $M$. } away from $K$ (ii) if convex bodies $K_1,\ldots, K_i, \ldots$ are almost disjoint and $\bigcup_{i=1}^\infty K_i = \mathbb{R}^{d+1}$ then 
\begin{equation}\label{e:WeightProperty1}
\sum_{i=1}^\infty w_{K_i} \lesssim 1. 
\end{equation} 
Also remark that $w_K(\mathbb{R}^{d+1}) := \int_{\R^{d+1}} w_K  \sim |K|$. 
We refer to \cite{BourDem16,FSWW} for the details on such weight function. 
In below the weight function $w_K$ satisfying the two properties might change from line to line or from left hand side of inequality to right hand side and we will not track the precise form of $w_k$.  
Then the (local) $\ell^2$-decoupling inequality for $E_{\mathbb{P}^d}$ states that 
\begin{equation}\label{e:DecoupLocal}
\| E_{\mathbb{P}^d} g \|_{L^p(B_{N^2})}
\lesssim_\varepsilon 
N^\varepsilon 
\big(
\sum_{\theta \in \mathcal{C}_{N^{-1}}} 
\| E_{\mathbb{P}^d} g_\theta \|_{L^p(w_{B_{N^2}})}^2
\big)^\frac12, 
\end{equation}
where $g_\theta:= \chi_\theta g$, $B_{N^2}$ is a ball in $\mathbb{R}^{d+1}$ radius $N^2$ and arbitrary centre, and 
$$
\| f \|_{L^p(w_{B_{N^2}})}
:=
\big(
\int_{\mathbb{R}^{d+1}}
|f(x,t)|^p w_{B_{N^2}}(x,t)\, dxdt
\big)^\frac1p.
$$
In fact, \eqref{e:DecoupLocal} directly follows from \eqref{e:DecoupFat} for the same $p$ since $ f = E_{\mathbb{P}^d} g \cdot w_{B_{N^2}}$ has a Fourier support on $\mathcal{N}_{N^{-2}}(\mathbb{P}^d)$, see the inequality (7) in \cite{BourDem15} for more details. 
Hence, as a corollary from Theorem \ref{t:BD15}, we have \eqref{e:DecoupLocal} for all $2\le p \le \frac{2(d+2)}{d}$. 
Contrary as is stated in \cite[Remark 5.2]{BourDem15} one can also recover \eqref{e:DecoupFat} from \eqref{e:DecoupLocal}, see also \cite[Theorem 5.1]{BourDem16} for the proof of it.  
Therefore, it suffices to show the following to prove Theorem \ref{theorem:ShellDecoupling}.  
	
\begin{theorem}\label{cl:ShellDecoup}
	Let $d\ge2$, $N\ge1$, $p = \frac{2(d+1)}{d-1}$ and take arbitrary $ d_* \in [1,2]$. 
	Then \eqref{e:DecoupLocal} holds for all $\varepsilon>0$ and all $g:\mathbb{R}^d \to \mathbb{C}$ satisfying 
	\begin{equation}\label{e:Assumpg}
	{\rm supp}\, g \subset \Bigl\{ \xi \in \mathbb{R}^d: d_* - \frac1N \le |\xi| \le d_* + \frac1{N} \Bigr\}. 
	\end{equation}
\end{theorem}

Below, we prove Theorem \ref{cl:ShellDecoup}. Our proof is based on the induction on scale argument which is now common technique in the Fourier restriction theory.  More precisely, we will reduce the matter to the lower dimensional decoupling inequality by exploiting the assumption \eqref{e:Assumpg} together with appropriate induction scheme. 
	In fact our strategy of the proof consists of basically three steps. First, after the decomposition of the physical space into small scale, we give a simple observation exploiting the assumption \eqref{e:Assumpg} and then  approximate each small pieces of shell by flat plates.  This enables us to appeal to the one dimension less Bourgain--Demeter's decoupling estimate at each small scale.  We then put together each estimate in a tight way by using the induction hypothesis.  To conclude the proof,  we do this procedure inductively. 
    We note that the idea of appealing to the lower dimensional estimate can be found in Bourgain--Guth's argument \cite{BourGuth11} for instance. We also mention the proof of the decoupling inequality for the cone \cite{BourDem15}, and for the moment curve \cite{BDG}. In addition,  such induction scheme,  namely,  applying something known at each small scale and then appealing to the induction hypothesis to make it global,  can be found in for instance Guth's simple proof of the non-endpoint multilinear Kakeya inequality \cite{Guth15}.  
	Also an idea of approximating nonlinear object by linear one via induction on scale argument can be found in the context of nonlinear Brascamp--Lieb inequalities, see \cite{BB,BBBCF,BBFS}. 
	
    Let us provide detailed argument.  Fix large $N\ge1$ and for each $1\le K\le N$, let $\mathcal{D}(K)$ be the best constant for the inequality 
	$$
	\| E_{\mathbb{P}^{d} } g \|_{L^p(B_{K^2})}
	\le 
	\mathcal{D}(K)
	\big(
	\sum_{\theta \in \mathcal{C}_{K^{-1}}} 
	\| E_{\mathbb{P}^{d} } g_\theta  \|_{L^p( w_{B_{K^2}} )}^2 
	\big)^\frac12 
	$$
	for all $g$ satisfying \eqref{e:Assumpg} and all $B_{K^2}$ whose radius $K^2$ and arbitrary centre. 
	From here and below in this section we always assume $ p = \frac{2(d+1)}{d-1}$.
	Our goal \eqref{e:DecoupLocal} can be read as $\mathcal{D}(N) \lesssim_\varepsilon N^\varepsilon$. 
	The key of the proof is the following. 
	\begin{proposition}\label{p:Induction}
	For any $\varepsilon>0$, there exists $C_\varepsilon>0$ such that 
	\begin{equation}\label{e:Goal2June}
		\mathcal{D}(K)\le C_\varepsilon K^\varepsilon \mathcal{D}(K^{\frac12})
	\end{equation}
	holds for all $K\in [1,N]$.
	\end{proposition}
	Once we could prove this, we may conclude $\mathcal{D}(N) \lesssim_\varepsilon N^\varepsilon$ as follows. 
	\begin{proof}[Proof of Proposition \ref{p:Induction} $\Rightarrow$ Theorem \ref{cl:ShellDecoup}] 
	Let $m\in \mathbb{N}$ be a large number which we will choose later. By applying \eqref{e:Goal2June} with $K = N, N^\frac12, N^{\frac1{2^2}},\ldots, N^{\frac1{2^m}}$ inductively, we obtain that 
	\begin{align*}
	\mathcal{D}(N)
	\le 
	C_\varepsilon N^\varepsilon \mathcal{D}(N^\frac12) 
	\le 
	C_\varepsilon^2 N^{\varepsilon(1+\frac12)} \mathcal{D}(N^{\frac1{2^2}})
	\le \ldots \le 
	C_\varepsilon^m N^{2\varepsilon } \mathcal{D}(N^{\frac1{2^m}}) 
	\end{align*}
	where we also used $ \sum_{i=1}^m 2^{-(i-1)} \le 2$ at the last inequality. 
	We now choose $m$ to be an unique integer which satisfies 
	$\log_2 \log\, N < m \le \log_2 \log\, N + 1$. In particular, we have  $N^{\frac1{2^m}} \le e$, Napier's number, and hence 
	$$
	\mathcal{D}(N)
	\le 
	C_\varepsilon^m N^{2\varepsilon } \mathcal{D}(e).  
	$$
	On the other hand, we clearly have $\sharp \mathcal{C}_{e^{-1}} \le 10^d$ and hence 
	$$
	\| E_{\mathbb{P}^{d} } g \|_{L^p( B_{e^2} )}
	\le 
	\sum_{\theta \in \mathcal{C}_{e^{-1}}} 
	\| E_{\mathbb{P}^{d} } g_\theta  \|_{L^p( B_{e^2} )}
	\le 
	10^{\frac{d}{2}} 
	\big( 
	\sum_{\theta \in \mathcal{C}_{e^{-1}}} 
	\| E_{\mathbb{P}^{d} } g_\theta  \|_{L^p( B_{e^2})}^2
	\big)^\frac12
	$$
	which means $\mathcal{D}(e) \lesssim1$. In the case of  $C_\varepsilon\le1$, this completes the proof. Suppose $C_\varepsilon>1$. 
	If we notice that 
	$$
	m - 1 \le \log_2 \log\, N = \frac{\log_{C_\varepsilon} \log\, N}{\log_{C_\varepsilon} 2}, 
	$$
	then it follows that 
	$$
	C_\varepsilon^m = C_\varepsilon C_\varepsilon^{m-1}
	\le 
	C_\varepsilon ( \log\, N)^{ \frac{1}{\log_{C_\varepsilon} 2} }
	\le 
	C_\varepsilon C_\varepsilon' N^\varepsilon
	$$
	which concludes $\mathcal{D}(N) \lesssim_\varepsilon N^{3\varepsilon}$. 
	\end{proof}
	
	Fix arbitrary $N\ge1$ and $\varepsilon>0$ and let us prove \eqref{e:Goal2June} in the following. 
	
	\subsection{Proof of Proposition \ref{p:Induction}: Apply the induction hypothesis}\label{subsection3.1}
		We begin with an application of the decoupling inequality at scale $\sqrt{K}$. 
		To this end we decompose $B_{K^2}$ into almost disjoint balls $(B_K^{(i)})_i$ with radius $K$: $B_{K^2} \subset \bigcup_i B_K^{(i)}$. 
		Then it follows from the definition of $\mathcal{D}(K^\frac12)$ that 
		$$
		\|
		E_{\mathbb{P}^d} g
		\|_{L^p(B_{K^2})}^p
		\lesssim 
		\sum_i
		\|
		E_{\mathbb{P}^d} g
		\|_{L^p(B_{K}^{(i)})}^p
		\le 
		\mathcal{D}(K^\frac12)^p
		\sum_i
		\big( 
		\sum_{\nu \in \mathcal{C}_{K^{-1/2}}}
		\|
		E_{\mathbb{P}^d} g_\nu
		\|_{L^p(w_{B_{K}^{(i)}})}^2
		\big)^\frac{p}2. 
		$$
		Since $p\ge2$ Minkowski's inequality and \eqref{e:WeightProperty1} show that 
		\begin{equation}\label{e:Step1}
		\| E_{\mathbb{P}^{d}} g \|_{ L^p( B_{K^2} ) }
		\lesssim 
		\mathcal{D}( K^{\frac12} ) \big( \sum_{ \nu \in \mathcal{C}_{K^{-1/2}} } \| E_{\mathbb{P}^{d}} g_\nu \|_{ L^p( w_{B_{K^2}} ) }^2  \big)^\frac12. 
		\end{equation}
		Our new observation is the following simple fact regarding ${\rm supp}\, g_\nu$. Namely the assumption \eqref{e:Assumpg} yields that 
		${\rm supp}\, g_\nu$ is contained in some 
		$ \frac{1}{K} \times \frac{1}{\sqrt{K}} \times \cdots \times \frac{1}{\sqrt{K}}$-plate as long as $K\le N$. 
		Note that this plate is more or less contained in the shell $\{ d_* - \frac1N \le |\xi| \le d_* + \frac1N \}$ and the thin direction is arbitrary.  
		With this in mind, we fix any $\nu \in \mathcal{C}_{K^{-1/2}}$ which is now regarded as $ \frac{1}{K} \times \frac{1}{\sqrt{K}} \times \cdots \times \frac{1}{\sqrt{K}}$-plate and then decouple $\| E_{\mathbb{P}^{d} } g_\nu \|_{L^p(\mathbb{R}^{d+1})} $ further. More precisely we next aim to show 
		\begin{equation}\label{e:Goal3June_1}
			\| E_{\mathbb{P}^{d} } g_\nu \|_{L^p( w_{B_{K^2}} )}
			\le 
			C_\varepsilon K^{\varepsilon} 
			\big( \sum_{\theta \in \mathcal{C}_{K^{-1}}: \theta \subset \nu } \| E_{\mathbb{P}^{d} } g_\theta \|_{L^p( w_{B_{K^2}} )}^2 \big)^\frac12. 
		\end{equation}
		Remark that weights $w_{B_{K^2}}$ on left and right hand side could be different. 
		By the rotation with respect to the $x$-variable, we may assume that 
		\begin{equation}\label{e:NuNov}
		\nu = \Bigl[ d_* - \frac1K, d_* + \frac1K \Bigr] \times \Bigl[-\frac1{\sqrt{K}}, \frac{1}{\sqrt{K}}\Bigr]^{d-1}. 
		\end{equation}
		In order to avoid heavy subscrips, we write $h = g_\nu$. 
		Moreover, since $w_{B_{K^2}}$ decays rapidly away from $B_{K^2}$ \eqref{e:Goal3June_1} can be reduced to show (see Proposition 9.15 in \cite{Dem})
			\begin{align}\label{e:Goal3June_2}
			&\| E_{\mathbb{P}^{d} } h \|_{L^p( B_{K^2} )}
			\le 
			C_\varepsilon K^\varepsilon 
			\big( \sum_{\theta \in \mathcal{C}_{K^{-1}}} \| E_{\mathbb{P}^{d} } h_\theta \|_{L^p( w_{B_{K^2}} ) }^2 \big)^\frac12,
			\end{align}
			for all $B_{K^2}$ and all $h$ such that  
			$$
			{\rm supp}\, h \subset  \Bigl[ d_* - \frac1K, d_* + \frac1K \Bigr] \times \Bigl[-\frac1{\sqrt{K}}, \frac{1}{\sqrt{K}}\Bigr]^{d-1}. 
			$$
			Once we could prove \eqref{e:Goal3June_2}, then we conclude \eqref{e:Goal2June}  by combining this with \eqref{e:Step1}.
			We will establish this by combining the dimension deduction argument in \cite{FSWW} and the justification of the locally constant heuristics in \cite{BourDem16}. 
			
			\subsection{Heuristic argument for \eqref{e:Goal3June_2}}\label{subsection3.2}	
			Let us first provide an idea of the proof of \eqref{e:Goal3June_2} as its proof involves several technical calculations. 
			Our model case emerges for the specific $h$ which is \textit{locally constant at scale $K^{-1}$} in $\xi_1$-variable in the sense that $h(\xi_1,\overline{\xi}) \sim h(\eta_1,\overline{\xi})$ holds for all $\xi_1,\eta_1 $ satisfying $|\xi_1 - \eta_1|\le K^{-1}$ and all $\overline{\xi} \in \mathbb{R}^{d-1}$. 
			To make argument simpler we also assume $B_{K^2} \sim [-K^2,K^2]^{d+1}$ and $d_* = 0$, in particular  
			$$
			{\rm supp}\, h \subset \Big[ -\frac1K, \frac1K \Big] \times \Big[-\frac1{\sqrt{K}}, \frac{1}{\sqrt{K}} \Big]^{d-1}. 
			$$
			For such $h$ we can easily reduce the dimension by 1 as 
			\begin{equation}\label{e:h_Const}
			E_{\mathbb{P}^d} h(x_1,\overline{x},t) 
			\sim 
			\big(
			\int_{-\frac1K}^\frac1K e^{i(x_1\xi_1 + t\xi_1^2)}
			\, d\xi_1
			\big)
			E_{\mathbb{P}^{d-1}} \overline{h}(\overline{x},t),
			\;\;\;
			(x_1,\overline{x},t) \in \mathbb{R}^{d+1}
			\end{equation}
			where $\overline{h}(\overline{\xi}) := h(0,\overline{\xi})$. 
			We then decompose 
			$$
			B_{K^2} \sim [-K^2,K^2]^{d+1} = \bigcup_{j_1} \big( [j_1K, (j_1+1)K] \times [-K^2,K^2]^{d} \big) =: \bigcup_{j_1} P_{j_1},
			$$ 
			where $j_1 $ runs over finite subset of integers. 
			Once we could prove 
			\begin{equation}\label{e:Heur1}
			\| E_{\mathbb{P}^d} h \|_{L^p( P_{j_1} )}
			\lesssim_\varepsilon K^\varepsilon 
			\big(
			\sum_{\theta \in \mathcal{C}_{K^{-1}}} 
			\| E_{\mathbb{P}^d} h_\theta \|_{L^p( w_{P_{j_1}} )}^2
			\big)^\frac12
			\end{equation}
			for all $j_1$ uniformly, then one can sum up to conclude \eqref{e:Goal3June_2}. 
			Let us demonstrate the proof of \eqref{e:Heur1} for $j_1=0$. 
			Notice that $(t,x) \in P_0$ implies $|x_1|\le K$ and $|t| \le K^2$ and hence \eqref{e:h_Const} indicates that 
			$$
			|E_{\mathbb{P}^d}h(x,t)| 
			\sim 
			K^{-1} 
			|E_{\mathbb{P}^{d-1}}\overline{h}(\overline{x},t)|
			$$
			on $P_0$. 
			Hence 
			$$
			\| E_{\mathbb{P}^d} h \|_{L^p(P_0)}
			\sim 
			K^{-1} 
			\| E_{\mathbb{P}^{d-1}} \overline{h}(\overline{x},t) \|_{L^p_{x,t}(P_0)}
			\sim 
			K^{-1+\frac1p} 
			\| E_{\mathbb{P}^{d-1}} \overline{h}  \|_{L^p_{\overline{x},t}([-K^2,K^2]^{d})}.
			$$
			Now recall that $ p = \frac{2(d+1)}{d-1}$ and this is the endpoint of Bourgain--Demeter's decoupling inequality on $\mathbb{R}^{d-1}$. 
			Therefore we obtain 
			$$
			\| E_{\mathbb{P}^{d-1}} \overline{h}  \|_{L^p_{\overline{x},t}([-K^2,K^2]^{d})}
			\lesssim_\varepsilon K^\varepsilon 
			\big(
			\sum_{\overline{\theta} \in \overline{\mathcal{C}}_{K^{-1}}}
			\| E_{\mathbb{P}^{d-1}} \overline{h}_{\overline{\theta}}  \|_{L^p_{\overline{x},t}( w_{ [-K^2,K^2]^{d}})}^2
			\big)^\frac12
			$$
			where $\overline{\theta}$ and $\overline{\mathcal{C}}_{K^{-1}}$ represent $d-1$ dimensional caps. 
			We then recall \eqref{e:h_Const} to see that 
			$$
			|E_{\mathbb{P}^{d-1}} \overline{h}_{\overline{\theta}}(\overline{x},t)|
			\sim 
			K 
			|E_{\mathbb{P}^{d}} {h}_{[- \frac1K,\frac1K]\times \overline{\theta}}(x,t)|
			$$
			for all $(x,t) \in P_0$. 
			This shows that 
			$$
			K^{-1+\frac1p}\| E_{\mathbb{P}^{d-1}} \overline{h}_{\overline{\theta}}  \|_{L^p_{\overline{x},t}( w_{ [-K^2,K^2]^{d}})}
			\sim 
			\| E_{\mathbb{P}^d} h_\theta \|_{L^p(w_{P_0})}
			$$
			with $\theta := [-\frac1K,\frac1K] \times \overline{\theta} \in \mathcal{C}_{K^{-1}}$ which concludes \eqref{e:Heur1} for $j_1=0$. 
			
			In the above argument, the critical heuristic is the locally constant property of $h$ and we need to justify it to make the proof rigorous. 
			Such situation often appears in the Fourier restriction theory, see the proof of an inequality (29) in  \cite{BourDem16}, also Bourgain--Guth's argument \cite{BourGuth11} and Guth's survey paper \cite{GuthSurvey}.   
			In fact we will employ the argument of Theorem 5.1 in \cite{BourDem16} to justify the above heuristic in the next subsection. 
			We also mention an inequality in \cite[(7.34)]{FSWW} which is the decoupling inequality for some thin plate.  This reminds our goal \eqref{e:Goal3June_2} but we need to be careful since the width of our plate is $K^{-1}$ whilst the physical scale is $K^2$.  As far as we understand the uncertainly principle would be applicable if the width of the plate is $K^{-2}$ when the physical scale is $K^2$. 
			Therefore one requires some non-trivial argument to justify the above heuristics.

			\subsection{Proof of Proposition \ref{p:Induction}: Rigorous argument for \eqref{e:Goal3June_2}}\label{subsection3.3}
	We begin with claiming that we may argue as if $d_* = 0$ by a nice change of variable which is standard in this context. We write $x = (x_1, \overline{x}) \in \R \times \R^{d-1}$ and $\xi = (\xi_1, \overline{\xi}) \in \R \times \R^{d-1}$. 
	By changing variable with respect to $\xi_1$, we have that  
		\begin{align*}
			E_{\mathbb{P}^{d}} h(x,t)
			= 
			e^{i (x_1 d_* + t|d_*|^2)} 
			E_{\mathbb{P}^{d}} [ \1_{[- \frac1K, \frac1K] \times [ -\frac{1}{\sqrt{K}}, \frac{1}{\sqrt{K}} ]^{d-1}} h^{d_*e_1}] (x_1 + 2d_* t, \overline{x},t),
		\end{align*}
		where 
		$$
		h^{d_*e_1}(\xi) := h(\xi_1+d_*, \overline{\xi}).
		$$
		
		Hence, 
		\begin{align*}
			\| E_{\mathbb{P}^{d}} h \|_{L^p( B_{K^2} )}^p
			=& 
			\int_{ B_{K^2} } 
			\bigl|E_{\mathbb{P}^{d}} [ \1_{[- \frac1K, \frac1K] \times [ -\frac{1}{\sqrt{K}}, \frac{1}{\sqrt{K}} ]^{d-1}} h^{d_*e_1}] (x_1 + 2d_* t, \overline{x},t) \bigr|^p\, dxdt. 
		\end{align*}
		By a further change of variables $y_1 = x_1 + 2d_* t$, $\overline{y}= \overline{x}$, $s= t$ whose Jacobian is $1$, we see that 
		\begin{align*}
			\| E_{\mathbb{P}^{d}} h \|_{L^p( B_{K^2} )}^p
			\le& 
			C
			\int_{ \widetilde{B_{K^2}} }  
			\bigl|E_{\mathbb{P}^{d}} [ \1_{[- \frac1K, \frac1K] \times [ -\frac{1}{\sqrt{K}}, \frac{1}{\sqrt{K}} ]^{d-1}} h^{d_*e_1}] (y,s)\bigr|^p\, dyds.
		\end{align*}
		Here $\widetilde{B_{K^2}}$ is an image of $B_{K^2}$ under the change of variable. This might not be a proper ball in general but one can cover it by $O(1)$ many balls with same size. Hence we regard $\widetilde{B_{K^2}}$ as another ball with radius $K^2$. 
		We next decompose the physical space into thin plates 
		$$
		\widetilde{B_{K^2}}
		\subset 
		\bigcup_{ j_1 } P_{ j_1 },
		\;\;\;
		P_{ j_1 } := [j_1K, (j_1+1)K) \times  \pi(\widetilde{B_{K^2}}), 
		$$
		where $\pi(\widetilde{B_{K^2}})$ denotes the projection of $\widetilde{B_{K^2}} \subset \mathbb{R}^{d+1}$ onto $\mathbb{R}^{d}$, and $j_1$ runs over finite subset of integers. 
		Note that $P_{ j_1 }$ is a thin plate in $x_1$-direction with width $K$.
		Then we have that  
		\begin{equation}\label{e:Step2}
		\| E_{\mathbb{P}^{d}} h \|_{L^p( B_{K^2})}
		\le 
		C \big( \sum_{ j_1 }  \big\| E_{\mathbb{P}^{d}} [ \1_{[- \frac1K, \frac1K] \times [ -\frac{1}{\sqrt{K}}, \frac{1}{\sqrt{K}} ]^{d-1}} h^{d_*e_1}] \big\|_{L^p_{y,s}(P_{ j_1 })}^p \big)^\frac1p.
		\end{equation}
		
		From now we fix arbitrary $j_1 $ and intend to show 
		\begin{align}\label{e:Goal4June_1}
		&  
		\big\| E_{\mathbb{P}^{d}} [ \1_{[- \frac1K, \frac1K] \times [ -\frac{1}{\sqrt{K}}, \frac{1}{\sqrt{K}} ]^{d-1}} h^{d_*e_1}] \big\|_{L^p_{y,s}(P_{ j_1})}  \nonumber \\
		\le& C_{\varepsilon,M} K^\varepsilon 
		\big(
		\sum_{\overline{\theta} \in \overline{\mathcal{C}}_{K^{-1}}}
		\big\| E_{\mathbb{P}^{d}} [ \1_{[- \frac1K, \frac1K] \times \overline{\theta}} h^{d_*e_1}] \big\|_{L^p_{y,s}( \mu_{ j_1 })}^2
		\big)^\frac12,
		\end{align}
		where 
		$$
		\mu_{ j_1 }(y,s) 
		:= 
		\Bigl(1+ \bigl| \frac{y_1 - j_1 K}{K} \bigr| \Bigr)^{-M} w_{ \pi(\widetilde{B_{K^2}}) }(\overline{y},s), 
		$$ 
		$M\gg1$ is some fixed large number, 
		and $ \overline{\mathcal{C}}_{K^{-1}}$ denotes a $d-1$ dimensional $\frac1K$-caps. 
		The weight function $\mu_{ j_1 }$ satisfies 
			\begin{equation}\label{e:Sumup}
			\sum_{j_1 } \mu_{j_1}(y,s) 
			\lesssim w_{\widetilde{B_{K^2}}}(y,s)
		\end{equation}
		by choosing $M\gg1$ large enough. 
		
		Suppose that we could prove \eqref{e:Goal4June_1} for a while and let us see how it yields \eqref{e:Goal3June_2}.
		In fact, it follows from \eqref{e:Step2},  Minkowski inequality, \eqref{e:Sumup}, and undoing the change of variable that 
		\begin{align*}
		&\| E_{\mathbb{P}^{d}} h \|_{L^p( B_{K^2})}\\
		\le& 
		C_\varepsilon K^\varepsilon  
		\bigg\{ \sum_{j_1 }
		\bigg(
		\sum_{\overline{\theta} \in \overline{\mathcal{C}}_{K^{-1}} }
		\big\| E_{\mathbb{P}^{d}} [ \1_{[- \frac1K, \frac1K] \times \overline{\theta}} h^{d_*e_1}] \big\|_{L^p_{y,s}( \mu_{ j_1 })}^2
		\bigg)^{\frac12\cdot p}
	       \bigg\}^\frac1p\\
	       \le& 
	       C_\varepsilon K^\varepsilon  
		\bigg\{ 
		\sum_{\overline{\theta} \in \overline{\mathcal{C}}_{K^{-1}}}
		\bigg(
		\sum_{j_1 }
		\big\| E_{\mathbb{P}^{d}} [ \1_{[- \frac1K, \frac1K] \times \overline{\theta}} h^{d_*e_1}] \big\|_{L^p_{y,s}( \mu_{ j_1})}^p
		\bigg)^{\frac1p\cdot 2}
	       \bigg\}^\frac12\\
	       \le&
	       C_\varepsilon K^\varepsilon  
		\big(
		\sum_{\overline{\theta} \in \overline{\mathcal{C}}_{K^{-1}}}
		\big\| E_{\mathbb{P}^{d}} [ \1_{[- \frac1K, \frac1K] \times \overline{\theta}} h^{d_*e_1}] \big\|_{L^p_{y,s}(  w_{\widetilde{B_{K^2}}})}^2
		\big)^{\frac12}\\
		\le& 
		C_\varepsilon K^\varepsilon  
		\big(
		\sum_{\overline{\theta} \in \overline{\mathcal{C}}_{K^{-1}} }
		\big\| E_{\mathbb{P}^{d}} [ \1_{[d_* - \frac1K, d_* + \frac1K] \times \overline{\theta}} h] \big\|_{L^p_{x,t}( w_{\widetilde{B_{K^2}}})}^2
		\big)^{\frac12}
		\end{align*}
		which yields the goal \eqref{e:Goal3June_2} since $[d_* - \frac1K, d_* + \frac1K] \times \overline{\theta}\in \mathcal{C}_{K^{-1}}$. 
		 
	Therefore the proof is boiled down to show \eqref{e:Goal4June_1}. 
	As we have explained heuristically we will achieve this by invoking the lower degree decoupling inequality.
	In order to justify the locally constant property of the Fourier transform, we invoke the argument using Taylor expansion; this argument can be found  in the proof of an inequality (9.5) in \cite{BourDem16} for instance. 
	To this end, we first shift the center of $P_{j_1}$ to the origin. 
	From now we denote the centre of $\pi( \widetilde{B_{K^2}})$ by $K^2(\overline{\mathbf{j}}, j_{d+1}) \in \mathbb{R}^{d}$ and write $\mathbf{j} = (j_1, \overline{\mathbf{j}}, j_{d+1})$. 
	We insert 
	$$
	1 = e^{-i (j_1 K \xi_1 + K^2 \overline{\mathbf{j}} \cdot \overline{\xi} + j_{d+1} K^2 |\xi|^2 )} e^{i  (j_1 K \xi_1 + K^2 \overline{\mathbf{j}} \cdot \overline{\xi} + j_{d+1} K^2 |\xi|^2 )}
	$$
	to see that 
	\begin{align*}
	&E_{\mathbb{P}^{d}} [ \1_{[- \frac1K, \frac1K] \times [ -\frac{1}{\sqrt{K}}, \frac{1}{\sqrt{K}} ]^{d-1}} h^{d_*e_1}](y,s) \\
	=:&
	\int_{-\frac{1}{K}}^{\frac{1}{K}} 
	e^{ i \big( (y_1 - j_1K)\xi_1 + (s - j_{d+1} K^2)|\xi_1|^2 \big) } \\
	&\times 
	\int_{ [ -\frac{1}{\sqrt{K}}, \frac{1}{\sqrt{K}} ]^{d-1} }
	e^{ i \big( (\overline{y} - \overline{\mathbf{j}} K^2) \cdot \overline{\xi} + (s - j_{d+1} K^2) |\overline{\xi}|^2 \big) }
	h^{d_*e_1}_{\xi_1, \mathbf{j}}(\overline{\xi})\, d\overline{\xi} d\xi_1,\\
	&{\rm where}\;\;\;
	h^{d_*e_1}_{\xi_1, \mathbf{j}}(\overline{\xi}):=  e^{i  (j_1 K \xi_1 + K^2 \overline{\mathbf{j}} \cdot \overline{\xi} + j_{d+1} K^2 |\xi|^2 )}h^{d_*e_1}(\xi)\\
	=& 
	\int_{-\frac{1}{K}}^{\frac{1}{K}} 
	e^{ i \big( (y_1 - j_1K)\xi_1 + (s - j_{d+1} K^2) |\xi_1|^2 \big) } 
	E_{\mathbb{P}^{d-1}}\big[ 
	h^{d_*e_1}_{\xi_1, \mathbf{j}}\big](\overline{y} - \overline{\mathbf{j}} K^2, s - j_{d+1} K^2)\,  d\xi_1.
	\end{align*}
	We do the Taylor expansion to the exponential part 
	\[
	\begin{split}
	& e^{ i \big( (y_1 - j_1K)\xi_1 + (s-j_{d+1} K^2) |\xi_1|^2 \big) }\\ 
	& = 
	\sum_{m,l=0}^\infty \frac{1}{m! l!} \big( \frac{10 i (y_1 - j_1 K)}{ K } \big)^m \big( \frac{10i (s - j_{d+1} K^2) }{K^2} \big)^l 
	\big( \frac{ K \xi_1}{10} \big)^m \big( \frac{ K^2|\xi_1|^2 }{10 } \big)^l. 
	\end{split}
	\]
	Since $(y,s) \in P_{ j_1 }$ implies $0\le y_1 - j_1 K \le K $ and $0\le s - j_{d+1} K^2 \le K^2$, we have that 
	\begin{align*}
	&\bigl|E_{\mathbb{P}^{d}} [ \1_{[- \frac1K, \frac1K] \times [ -\frac{1}{\sqrt{K}}, \frac{1}{\sqrt{K}} ]^{d-1}} h^{d_*e_1}](y,s) \bigr| \\
	\le& 
	\sum_{m,l=0}^\infty \frac{10^m 10^l}{m! l!}
	\Big| 
	\int_{-\frac{1}{K}}^{\frac{1}{K}} 
	\big( \frac{K \xi_1}{10} \big)^m \big( \frac{ K^2 |\xi_1|^2 }{10 } \big)^l
	E_{\mathbb{P}^{d-1}}\big[ 
	h^{d_*e_1}_{\xi_1, \mathbf{j}}\big](\overline{y} - \overline{\mathbf{j}} K^2, s - j_{d+1} K^2)\,  d\xi_1
	\Big|\\
	=:& 
	\sum_{m,l=0}^\infty \frac{10^m 10^l}{m! l!}
	\big|
	E_{\mathbb{P}^{d-1}} 
	\big[
	h^{d_*e_1}_{\mathbf{j},m,l}
	\big](\overline{y} - \overline{\mathbf{j}} K^2, s - j_{d+1} K^2)
	\big|
	\end{align*}
	for all $(y,s) \in P_{ j_1 }$, where 
	$$
	h^{d_*e_1}_{\mathbf{j},m,l}(\overline{\xi})
	:= 
	\int_{-\frac{1}{K}}^{\frac{1}{K}} 
	\big( \frac{K \xi_1}{10} \big)^m \big( \frac{ K^2 |\xi_1|^2 }{10 } \big)^l
	h^{d_*e_1}_{\xi_1, \mathbf{j}}(\overline{\xi})\,  d\xi_1.
	$$

	From this, $| [j_1K, (j_1+1)K] | = K$ and $P_{ j_1 } = [j_1K, (j_1+1)K] \times \pi( \widetilde{B_{K^2}} ) $, we obtain that 
	\begin{align*}
	&\big\| E_{\mathbb{P}^{d}} [ \1_{[- \frac1K, \frac1K] \times [ -\frac{1}{\sqrt{K}}, \frac{1}{\sqrt{K}} ]^{d-1}} h^{d_*e_1}] \big\|_{L^p_{y,s}(P_{ j_1 })}\\
	\le& 
	\sum_{m,l=0}^\infty \frac{10^m 10^l}{m! l!}
	K^\frac1p
	\big\| 
	E_{\mathbb{P}^{d-1}} 
	\big[
	h^{d_*e_1}_{\mathbf{j},m,l}
	\big]
	(\overline{y} - \overline{\mathbf{j}} K^2, s - j_{d+1} K^2)
	\big\|_{ L^p_{\overline{y},s}( \pi( \widetilde{B_{K^2}} ) ) }\\
	\sim& 
	\sum_{m,l=0}^\infty \frac{10^m 10^l}{m! l!}
	K^\frac1p
	\big\| 
	E_{\mathbb{P}^{d-1}} 
	\big[
	h^{d_*e_1}_{\mathbf{j},m,l}
	\big]
	\big\|_{ L^p_{\overline{y},s}( [0, K^2]^{d} ) }. 
	\end{align*}
	
	At this stage we may appeal to the \eqref{e:DecoupLocal} in $\mathbb{R}^{d}$ to see that 
	\begin{align}\label{e:Step3}
	&\big\| E_{\mathbb{P}^{d}} [ \1_{[- \frac1K, \frac1K] \times [ -\frac{1}{\sqrt{K}}, \frac{1}{\sqrt{K}} ]^{d-1}} h^{d_*e_1}] \big\|_{L^p_{y,s}(P_{ j_1})}\nonumber \\
	\le &
	C_\varepsilon K^\varepsilon K^\frac1p
	\sum_{m,l=0}^\infty \frac{10^m 10^l}{m! l!}
	\big( 
	\sum_{\overline{\theta} \in \overline{C}_{K^{-1}}}
	\big\| 
	E_{\mathbb{P}^{d-1}} 
	\big[
	\1_{\overline{\theta}}
	h^{d_*e_1}_{\mathbf{j},m,l}
	\big]
	\big\|_{ L^p_{\overline{y},s}(w_{ [0,K^2]^{d} }) }^2
	\big)^\frac12
	\end{align}
	since $ p = \frac{2(d+1)}{d-1}$ is the endpoint of Bourgain--Demeter's decoupling inequality on $\mathbb{R}^{(d-1) + 1}$.
	With our goal \eqref{e:Goal4June_1} in mind, 
    we next reproduce $E_{\mathbb{P}^{d}}$ from $E_{\mathbb{P}^{d-1}}$. 
	This step is also the consequence from the locally constant property and one can  find the following justifying argument in the proof of (9.5) in \cite{BourDem16} again. 
	Let us focus on each pieces by fixing $m,l,\theta$. 
	We now take arbitrary $y_1 \in I_{j_1}(K):= [j_1K, (j_1+1)K]$ and insert a simple identity 
	$
	1 = e^{- i(y_1 \xi_1 + s|\xi_1|^2)} e^{i(y_1 \xi_1 + s|\xi_1|^2)},
	$
	to see that 
	\begin{align*}
	&E_{\mathbb{P}^{d-1}} 
	\big[
	\1_{\overline{\theta}}
	h^{d_*e_1}_{\mathbf{j},m,l}\big]
	(\overline{y},s) \\
	=& 
	\int_{-\frac{1}{K}}^{\frac{1}{K}} \int_{\overline{\theta}} 
	\big( \frac{K \xi_1}{10} \big)^m \big( \frac{ K^2 |\xi_1|^2 }{10 } \big)^l
	e^{ i(\overline{y}\cdot \overline{\xi} + s|\overline{\xi}|^2)} 
	h^{d_*e_1}_{\xi_1,\mathbf{j}}(\overline{\xi})  \,  d\xi_1d\overline{\xi}\\
	=& 
	\int_{-\frac{1}{K}}^{\frac{1}{K}} \int_{\overline{\theta}} 
	\big( \frac{K \xi_1}{10} \big)^m \big( \frac{ K^2 |\xi_1|^2 }{10 } \big)^l
	e^{- i(y_1 \xi_1 + s|\xi_1|^2)} 
	e^{ i( {y}\cdot \xi + s|{\xi}|^2 ) }
	e^{ij_1K\xi_1}
	h^{d_*e_1}_{(\overline{\mathbf{j}},j_{d+1})}({\xi})\,  d\xi\\
	&{\rm where}
	\;\;\;
	h^{d_*e_1}_{(\overline{\mathbf{j}},j_{d+1})}({\xi})
	:=
	e^{ iK^2 \big(\overline{\mathbf{j}} \cdot \overline{\xi} + j_{d+1} |\xi|^2  \big) } h^{d_* e_1}(\xi)
	\\
	=& 
	\int_{-\frac{1}{K}}^{\frac{1}{K}} \int_{\overline{\theta}} 
	\big( \frac{ K \xi_1}{10} \big)^m \big( \frac{ K^2 |\xi_1|^2 }{10 } \big)^l
	e^{- i((y_1- j_1K) \xi_1 + s|\xi_1|^2)} 
	e^{ i( {y}\cdot \xi + s|{\xi}|^2 ) }
	h^{d_*e_1}_{(\overline{\mathbf{j}},j_{d+1})}({\xi})\,  d\xi\\
	=& 
	\sum_{m',l'=0}^\infty 
	\frac{1}{m'!l'!}
	\big( \frac{ 10 i (y_1 - j_1K) }{K} \big)^{m'}
	\big(
	\frac{10i s}{K^2}
	\big)^{l'} \\
	& \qquad
	\int_{-\frac{1}{K}}^{\frac{1}{K}} \int_{\overline{\theta}} 
	\big( \frac{ K \xi_1}{10} \big)^{m+m'} \big( \frac{ K^2 |\xi_1|^2 }{10 } \big)^{l+l'}
	e^{ i( {y}\cdot \xi + s|{\xi}|^2 ) }
	h^{d_*e_1}_{(\overline{\mathbf{j}},j_{d+1})}({\xi})\,  d\xi.
	\end{align*}
	Since $y _1 \in I_{j_1}(K)$ is arbitrary, we obtain that for all $s \in [0, K^2]$, 
	\begin{align*}
	&|E_{\mathbb{P}^{d-1}} 
	\big[
	\1_{\overline{\theta}}
	h^{d_*e_1}_{\mathbf{j},m,l}\big]
	(\overline{y},s) |\\
	\le& 
	\inf_{y_1 \in I_{j_1}(K)}
	\sum_{m',l'=0}^\infty 
	\frac{10^{m'} 10^{l'}}{m'!l'!}
	\Bigl|
	\int_{-\frac{1}{K}}^{\frac{1}{K}} \int_{\overline{\theta}} 
	\mathcal{M}(\xi_1)
	e^{ i( {y}\cdot \xi + s|{\xi}|^2 ) }
	h^{d_*e_1}_{(\overline{\mathbf{j}},j_{d+1})}({\xi})\,  d\xi
	\Bigr|,
	\end{align*}
	where $\mathcal{M}(\xi_1):= \big( \frac{ K \xi_1}{10} \big)^{m+m'} \big( \frac{ K^2 |\xi_1|^2 }{10 } \big)^{l+l'}$.
	This shows that 
	\begin{align}\label{e:Mayonaka}
	&\big\| 
	E_{\mathbb{P}^{d-1}} 
	\big[
	\1_{\overline{\theta}}
	h^{d_*e_1}_{\mathbf{j},m,l}
	\big]
	\big\|_{ L^p_{\overline{y},s}( w_{[0,K^2]^{d}}) } \nonumber \\
	\lesssim& 
	\frac{1}{K^{\frac1p}}
	\big\| 
	E_{\mathbb{P}^{d-1}} 
	\big[
	\1_{\overline{\theta}}
	h^{d_*e_1}_{\mathbf{j},m,l}
	\big](\overline{y},s)
	\big\|_{ L^p_{{y},s}(   w_{[0,K^2]^{d}} \otimes w_{I_{j_1}(K)}  ) } \nonumber\\
	\le& 
	\frac{1}{K^{\frac1p}}
	\sum_{m',l'=0}^\infty 
	\frac{10^{m'} 10^{l'}}{m'!l'!} \nonumber \\
	&\times 
	\bigg\| 
	\int_{[-\frac{1}{K},\frac{1}{K}]\times \overline{\theta}} 
	\mathcal{M}(\xi_1)
	e^{ i( {y}\cdot \xi + s|{\xi}|^2 ) }
	h^{d_*e_1}_{(\overline{\mathbf{j}},j_{d+1})}({\xi})\,  d\xi
	\bigg\|_{ L^p_{{y},s}(  w_{[0,K^2]^{d}} \otimes w_{I_{j_1}(K)} ) }\nonumber \\
	=& 
	\frac{1}{K^{\frac1p}}
	\sum_{m',l'=0}^\infty 
	\frac{10^{m'} 10^{l'}}{m'!l'!}\\
	&\times 
	\bigg\| 
	\int_{[-\frac{1}{K},\frac{1}{K}]\times \overline{\theta}} 
	\mathcal{M}(\xi_1)
	e^{ i( {y}\cdot \xi + s|{\xi}|^2 ) }
	h^{d_*e_1}_{\mathbf{j}}({\xi})\,  d\xi
	\bigg\|_{ L^p_{{y},s}(   w_{[0,K^2]^{d}} \otimes w_{[0,K)}   ) } \nonumber \\
	&{\rm where}\;\;\; 
	h^{d_*e_1}_{\mathbf{j}}({\xi})
	:= 
	e^{ i \big(Kj_1\xi_1 + K^2 ( \overline{\mathbf{j}},j_{d+1} ) \cdot (\overline{\xi}, |\xi|^2) \big)} h^{d_*e_1}(\xi).\nonumber 
	\end{align}
	We next intend to get rid of the Fourier multiplier 
	$$ 
	\1_{[-\frac1K, \frac1K]}(\xi_1)  \big( \frac{ K \xi_1}{10} \big)^{m+m'} \big( \frac{ K^2 |\xi_1|^2 }{10 } \big)^{l+l'}
	= 
	\Big[ \1_{[-1, 1]}  \frac{ \star^{m+m'+2(l+l')} }{10^{m+m'+l+l'}} \Big]( K\xi_1)
	$$ 
	from the above expression uniformly in $m,m',l,l'$.  
	To this end, we first notice that 
	\begin{align*}
	\Big|
	\big( \frac{d}{d\xi_1} \big)^M \big( \frac{ \star^{m+m'+2(l+l')} }{10^{m+m'+l+l'}} \big)(\xi_1)
	\Big|
	\le& 
	|m+m'+2(l+l')|^M 10^{-(m+m'+l+l')}\\
	\le &
	2^M 
	|m+m'+ l+l'|^M 10^{-(m+m'+l+l')}\\
	\le& 
	C_M
	\end{align*}
	for all $\xi_1 \in [-1,1]$ and any $M\gg1$\footnote{We here used that $ \sup_{t \in [0,\infty)} t^M 10^{-t} \le \exists C_M$}. 
	
	Hence we may extend the multiplier smoothly, namely there exists $\mathfrak{m}_{m,m',l,l'} \in C^\infty([-2,2])$ such that 
	$$
	\mathfrak{m}_{m,m',l,l'} = \frac{ \star^{m+m'+2(l+l')} }{10^{m+m'+l+l'}}\;\;\;{\rm on} \;\;\; [-1,1]
	$$
	and that 
	\begin{equation}\label{e:DeriMulti}
		\big\|\big( \frac{d}{d\xi_1} \big)^M \mathfrak{m}_{m,m',l,l'} \big\|_{L^\infty(\mathbb{R})} \le C_M.
	\end{equation}
	Note that \eqref{e:DeriMulti} yields that 
	\begin{equation}\label{e:DecayMultiplier}
	\begin{split}
		| {}^{\vee} \mathfrak{m}_{m,m',l,l'}(y_1) |
		& =
		\Big|
		\frac{1}{(iy_1)^M}
		\int_{-2}^2 e^{i y_1 \xi_1} \big( \frac{d}{d\xi_1}\big)^M \mathfrak{m}_{m,m',l,l'}(\xi_1)\, d\xi_1
		\Big|\\
		& \le C_M ( 1 + |y_1| )^{-M}. 
		\end{split}
	\end{equation}
	
	Using this, we go back to \eqref{e:Mayonaka} and obtain that 
	\begin{align*}
	&\bigg\| 
	\int_{[-\frac{1}{K},\frac{1}{K}]\times \overline{\theta}} 
	\mathcal{M}(\xi_1)
	e^{ i( {y}\cdot \xi + s|{\xi}|^2 ) }
	h^{d_*e_1}_{\mathbf{j}}({\xi})\,  d\xi
	\bigg\|_{ L^p_{{y},s}(w_{[0,K^2]^{d}} \otimes w_{[0,K)}   ) }\\
	=& 
	\bigg\| 
	\int_{[-\frac{1}{K},\frac{1}{K}]\times \overline{\theta}} 
	\mathfrak{m}_{m,m',l,l'}(K \xi_1) 
	e^{ i( {y}\cdot \xi + s|{\xi}|^2 ) }
	h^{d_*e_1}_{\mathbf{j}}({\xi})\,  d\xi
	\bigg\|_{ L^p_{{y},s}(w_{[0,K^2]^{d}} \otimes w_{[0,K)}  ) } \\
	=& 
	\bigg\| 
	\big[ \frac1K {}^{\vee} {\mathfrak{m}_{m,m',l,l'}} (\frac{\cdot}{K}) \big] \ast_{y_1} 
	E_{\mathbb{P}^{d}} \big[ \1_{ [-\frac{1}{K},\frac{1}{K}]\times \overline{\theta} }
	h^{d_*e_1}_{\mathbf{j}} \big]
	\bigg\|_{ L^p_{{y},s}( w_{[0,K^2]^{d}}  \otimes w_{[0,K)}  ) }.
	\end{align*}
	
	In view of \eqref{e:DecayMultiplier}, we know that 
	$$
	\Bigl\| \frac1K {}^{\vee} {\mathfrak{m}_{m,m',l,l'}} (\frac{\cdot}{K}) \Bigr\|_{L^1(\mathbb{R})} \le C_M 
	$$
	from which together with H\"{o}lder's inequality we obtain that 
	\begin{align*}
		&\big| 
		\big[ \frac1K {}^{\vee} {\mathfrak{m}_{m,m',l,l'}} (\frac{\cdot}{K}) \big] \ast_{y_1} 
		E_{\mathbb{P}^{d}} \big[ \1_{ [-\frac{1}{K},\frac{1}{K}]\times \overline{\theta} }
		h^{d_*e_1}_{\mathbf{j}} \big](y, s)
		\big|\\
		=& 
		\Big|
		\int  
		E_{\mathbb{P}^{d}} \big[ \1_{ [-\frac{1}{K},\frac{1}{K}]\times \overline{\theta} }
		h^{d_*e_1}_{\mathbf{j}} \big] (y_1-z,\overline{y},s)
		\big[ \frac1K {}^{\vee} {\mathfrak{m}_{m,m',l,l'}} (\frac{\cdot}{K}) \big] (z)
		\, dz
		\Big|\\
		\leq & 
		\int  
		\bigl| 
		E_{\mathbb{P}^{d}} \big[ \1_{ [-\frac{1}{K},\frac{1}{K}]\times \overline{\theta} }
		h^{d_*e_1}_{\mathbf{j}} \big] (y_1-z,\overline{y},s)
		\bigr|
		\, dP(z) C_M
		\\
		&{\rm where} \;\;\; 
		dP(z) := \frac{1}{C_M}\big| \frac1K {}^{\vee} {\mathfrak{m}_{m,m',l,l'}} (\frac{\cdot}{K}) \big| (z)
		\, dz \\ 
		\le& 
		C_M
		\big(
		\int  
		\bigl|
		E_{\mathbb{P}^{d}} \big[ \1_{ [-\frac{1}{K},\frac{1}{K}]\times \overline{\theta} }
		h^{d_*e_1}_{\mathbf{j}} \big] (y_1-z,\overline{y},s)
		\bigr|^p
		\, dP(z)
		\big)^\frac1p
		\big(
		\int 
		\, dP(z)
		\big)^{\frac1{p'}}\\
		\le& 
		C_M 
		\Bigl( 
		\big|
		 \frac1K {}^{\vee} {\mathfrak{m}_{m,m',l,l'}} (\frac{\cdot}{K})
		\big|
		\ast_{y_1} 
		\big|
		\big[E_{\mathbb{P}^{d}} \big[ \1_{ [-\frac{1}{K},\frac{1}{K}]\times \overline{\theta} }
		h^{d_*e_1}_{\mathbf{j}} \big]
		\big|^p
		\Bigr)^{\frac{1}{p}} (y, s)
	\end{align*}
	Hence we see that 
	\begin{align}\label{e:Mayonaka2}
	&\bigg\| 
	\int_{[-\frac{1}{K},\frac{1}{K}]\times \overline{\theta}} 
	\mathcal{M}(\xi_1)
	e^{ i( {y}\cdot \xi + s|{\xi}|^2 ) }
	h^{d_*e_1}_{\mathbf{j}}({\xi})\,  d\xi
	\bigg\|_{ L^p_{{y},s}(  w_{[0,K^2]^{d}} \otimes w_{[0,K)} )} \nonumber\\
	\le&	
	C_M
	\bigg\| 
		\Bigl|
		 \frac1K {}^{\vee} {\mathfrak{m}_{m,m',l,l'}} (\frac{\cdot}{K})
		\Bigr|
		\ast_{y_1} 
		\Bigl|
		\big[E_{\mathbb{P}^{d}} \big[ \1_{ [-\frac{1}{K},\frac{1}{K}]\times \overline{\theta} }
		h^{d_*e_1}_{\mathbf{j}} \big]
		\Bigr|^p
	\bigg\|_{ L^1_{{y},s}( w_{[0,K^2]^{d}} \otimes w_{[0,K)} )}  ^\frac1p.
	\end{align}
	If we write $\phi(y_1) := \big| E_{\mathbb{P}^{d}} \big[ \1_{ [-\frac{1}{K},\frac{1}{K}]\times \overline{\theta} }
	h^{d_*e_1}_{\mathbf{j}}\big] \big|^p(y,s) $ for simplicity, then 
    
    \begin{align*}
		&\Bigl\| \big|\frac1K {}^{\vee} {\mathfrak{m}_{m,m',l,l'}} (\frac{\cdot}{K}) \big| \ast_{y_1} \phi \Bigr\|_{L^1(w_{[0,K]})} \\
		\le&
		C_M 
		\int \phi(y_1) \, d\mu_{[-K,K]}(y_1),\;\;\;
		d\mu_{[-K,K]}(y_1):= \bigl( 1+ \bigl| \frac{y_1}{K} \bigr| \bigr)^{-M}\, dy_,
	\end{align*}
    since we have from \eqref{e:DecayMultiplier} that 
    \begin{align*}
	&\int\big|\frac1K {}^{\vee} {\mathfrak{m}_{m,m',l,l'}} (\frac{z-y_1}{K}) \big|w_{[0,K]}(z)\, dz\\
		\le& 
		C_M \int \bigl(1+ \bigl| \frac{y_1}{K} -z \bigr| \bigr)^{-M} w_{[0,1]}(z)\, dz
		\sim 
		C_M \bigl( 1+ \bigl|  \frac{y_1}{K} \bigr| \bigr)^{-M}. 
	\end{align*}

	Therefore returning to \eqref{e:Mayonaka2}, this shows that 
	\begin{align*}
	&\bigg\| 
	\int_{[-\frac{1}{K},\frac{1}{K}]\times \overline{\theta}} 
	\mathcal{M}(\xi_1)
	e^{ i( {y}\cdot \xi + s|{\xi}|^2 ) }
	h^{d_*e_1}_{\mathbf{j}}({\xi})\,  d\xi
	\bigg\|_{ L^p_{{y},s}(  w_{[0,K^2]^{d}} \otimes w_{[0,K)}   ) }\\
	\le& 
	C_M 
	\bigg\| 
	\big| 
	E_{\mathbb{P}^{d}} \big[ \1_{ [-\frac{1}{K},\frac{1}{K}]\times \overline{\theta} }
	h^{d_*e_1}_{\mathbf{j}} \big]\big|^p 
	\bigg\|_{ L^1_{{y},s}( \mu_{   [0,K^2]^{d} \times [-K,K] } ) }^\frac1p\\
	&{\rm where}\;\;\; 
	\mu_{  [0,K^2]^{d} \times [-K,K]  }(y,s)
	:= 
	(1+ | \frac{y_1}{K} |)^{-M} w_{[0,K^2]^{d}}(\overline{y},s)
	\\
	=& 
	C_M 
	\big\|  
	E_{\mathbb{P}^{d}} \big[ \1_{ [-\frac{1}{K},\frac{1}{K}]\times \overline{\theta} }
	h^{d_*e_1}_{\mathbf{j}} \big]
	\big\|_{ L^p_{{y},s}( \mu_{  [0,K^2]^{d} \times [-K,K] } ) }.
	\end{align*}
	Overall we could manage to get rid of the Fourier multiplier uniformly $m,m',l,l',j_1$. 
	We combine this estimate and \eqref{e:Mayonaka} to see that 
	\begin{align*}
	&\big\| 
	E_{\mathbb{P}^{d-1}} 
	\big[
	\1_{\overline{\theta}}
	h^{d_*e_1}_{\mathbf{j},m,l}
	\big]
	\big\|_{ L^p_{\overline{y},s}(w_{[0,K^2]^{d}}) }\nonumber \\
	\le& 
	C_M \frac{1}{K^\frac1p} \big\|  
	E_{\mathbb{P}^{d}} \big[ \1_{ [-\frac{1}{K},\frac{1}{K}]\times \overline{\theta} }
	h^{d_*e_1}_{\mathbf{j}} \big]
	\big\|_{ L^p_{{y},s}( \mu_{   [0,K^2]^{d} \times [-K,K] } ) }\\
	=& 
	C_M \frac{1}{K^{\frac1p}}
	\big\|
	E_{\mathbb{P}^{d}}\big[ \1_{ [ - \frac1K,\frac1K ]\times \overline{\theta} } h^{d_*e_1} \big] 
	\big\|_{ L^p_{{y},s}( \mu_{j_1} ) }
	\end{align*}

	Inserting this to  \eqref{e:Step3}, we conclude that 
	\begin{align*}
	&
	\big\| E_{\mathbb{P}^{d}} [ \1_{[- \frac1K, \frac1K] \times [ -\frac{1}{\sqrt{K}}, \frac{1}{\sqrt{K}} ]^{d-1}} h^{d_*e_1}] \big\|_{L^p_{y,s}(P_{ j_1 })} \\
	\le&
	C_\varepsilon K^\varepsilon K^\frac1p
	\sum_{m,l=0}^\infty \frac{10^m 10^l}{m! l!}
	\big( 
	\sum_{\overline{\theta} \in \overline{\mathcal{C}}_{K^{-1}}}
	\big\| 
	E_{\mathbb{P}^{d-1}} 
	\big[
	\1_{\overline{\theta}}
	h^{d_*e_1}_{ \mathbf{j},m,l}
	\big]
	\big\|_{ L^p_{\overline{y},s}( w_{ [0,K^2]^d  }) }^{ 2}
	\big)^\frac12\\
	\le &
	C_{\varepsilon,M} K^\varepsilon
	\big( 
	\sum_{\overline{\theta} \in \overline{\mathcal{C}}_{K^{-1}}}
	\big\|
	E_{\mathbb{P}^{d}}\big[ \1_{ [ - \frac1K,\frac1K ]\times \overline{\theta} } h^{d_*e_1} \big] 
	\big\|_{ L^p_{{y},s}( \mu_{ j_1} ) }^2
	\big)^\frac12
	\end{align*}
	which is our goal \eqref{e:Goal4June_1} and this completes the proof. 
	
            We end this section by giving a slight generalisation of Theorem \ref{cl:ShellDecoup} which is indeed necessary for the application to our PDE problem.
            At the same time, this point explains a difference between Theorem \ref{theorem:ShellDecoupling} and the work by Guo--Zorin-Kranich \cite{GZ}. 
            As we briefly mentioned in the introduction, Guo--Zorin-Kranich also considered the decoupling inequality with some geometrical constraint on ${\rm supp}\, \hat{f}$ or equivalently ${\rm supp}\, g$. 
            More precisely, they observed in Theorem 2.5 in \cite{GZ} that, in the framework of $\ell^p$-decoupling inequality rather than $\ell^2$-decoupling inequality, if one could prove the decoupling inequality for some exponent $p$, for all hyperplanes in $\mathbb{R}^d$, and all $g: \{\xi \in \mathbb{R}^d:|\xi|\le2\} \to \mathbb{C}$ whose support is contained in $N^{-1}$-neighborhood of the hyperplane, then one can upgrade it to the decoupling inequality for $g$ whose support is contained in $N^{-1}$-neighborhood of $C^2$ compact hypersurface. 
            One can regard our reduction of Theorem \ref{cl:ShellDecoup} to \eqref{e:Goal3June_2} as the $\ell^2$-decoupling version of this fact. 
            The proof of Theorem 2.5 in \cite{GZ} is based on the induction on the scale of the curvature of the hypersurface supporting ${\rm supp}\, g$ and Guo--Zorin-Kranich appealed to the parabolic rescaling at each inductive step.  
            On the other hand, our argument in Subsection \ref{subsection3.1} is based on the simple observation on the geometry of ${\rm supp}\, g_\nu$ just after \eqref{e:Step1} and hence it is a parabolic rescaling free argument. 
            For the purpose of the application to our PDE problem, this difference is important. 
            To be more precise, it is worth to clarify where the loss $K^\varepsilon$ in \eqref{e:Goal3June_2} comes from. 
            In fact, in our argument, it comes from the loss in Bourgain--Demeter's $\ell^2$-decoupling inequality in $\mathbb{R}^{d-1}$ only. 
            Namely, if $\mathcal{D}_{d-1}(K)$ denotes the best constant for the inequality 
            $$
            \| E_{\mathbb{P}^{d-1} } g \|_{L^p(B_{K^2})}
	        \le 
	        \mathcal{D}_{d-1}(K)
	        \big(
	        \sum_{\overline{\theta} \in \overline{\mathcal{C}}_{K^{-1}}} 
	        \| E_{\mathbb{P}^{d-1} } g_{\overline{\theta}}  \|_{L^p( w_{B_{K^2}} )}^2 
	        \big)^\frac12 
            $$
            for all reasonable $g : (-10,10)^{d-1} \to \mathbb{C}$ and all $B_{K^2} \subset \mathbb{R}^{d-1}\times \mathbb{R}$, then our proof shows that \eqref{e:Goal3June_2} can be stated as 
            \begin{equation}\label{e:MorePrecise}
            \| E_{\mathbb{P}^{d} } h \|_{L^p( B_{K^2} )}
			\le 
			C \mathcal{D}_{d-1}(K) 
			\big( \sum_{\theta \in \mathcal{C}_{K^{-1}}} \| E_{\mathbb{P}^{d} } h_\theta \|_{L^p( w_{B_{K^2}} ) }^2 \big)^\frac12. 
            \end{equation} 
            The $\ell^p$-decoupling version of this inequality can be found in Theorem 2.2 in Guo--Zorin-Kranich \cite{GZ} but they lose an extra factor of $\log K$. 
            Hence, up to a difference between $\ell^2$ and $\ell^p$-decoupling, this inequality improves Theorem 2.2 in Guo--Zorin-Kranich \cite{GZ} for the specific case\footnote{To be fair, we note that the argument in Guo--Zorin-Kranich \cite{GZ} handles much wider class of geometric constraint and $\log\, K$ factor is not so relevant for their purpose.}. 
            By virtue of this improvement, we can slightly generalise Theorem \ref{cl:ShellDecoup}.
            \begin{corollary}\label{Cor:TwoScaleDecoup}
                Let $d\ge 2$, $N_1 \ge N_2\ge1$, $p=\frac{2(d+1)}{d-1}$ and take arbitrary $d_* \in[1,2]$. 
                Then for arbitrary small $\varepsilon$, 
                \begin{equation}\label{e:N_2Gain}
                \big\| E_{\mathbb{P}^d} g \big\|_{L^p(B_{N_1^2})} \le C_\varepsilon N_2^{\varepsilon} \big( \sum_{\theta \in \mathcal{C}_{N_1^{-1}}} \big\| E_{\mathbb{P}^d} g_\theta  \big\|_{L^p(w_{B_{N_1^2}})}^2  \big)^\frac12
                \end{equation}
                holds for all $g$ satisfying 
                \begin{equation}\label{e:Assumpg2}
                {\rm supp}\, g\subset \{\xi\in\mathbb{R}^d: d_*-\frac1{N_1}\le |\xi|\le d_* + \frac1{N_1}\} \cap B_{N_2/N_1}
                \end{equation}
                where $B_{N_2/N_1} \subset \mathbb{R}^d$ is a ball of radius $N_2/N_1 \le 1$ with arbitrary centre. 
            \end{corollary}

            \begin{proof}
            Note that if $N_2^\varepsilon$ in \eqref{e:N_2Gain} is replaced by $N_1^\varepsilon$, then the inequality is no more than  Theorem \ref{cl:ShellDecoup}. 
            In this sense, the point here is to give an improvement $N_2^\varepsilon$ in the case $N_2\ll N_1$.
            When $N_1\ge N_2\ge N_1^{1/2}$, \eqref{e:N_2Gain} follows from Theorem \ref{cl:ShellDecoup} and hence it suffices to consider the case $N_2\le N_1^{1/2}$. 
            The argument in this case is almost parallel to the one we gave in the above but simpler. In fact we will use \eqref{e:MorePrecise} at one scale $K = N_1$ and we do not need to induct. 
            First, because of the rotation invariance, we may assume that the centre of $B_{N_2/N_1}$ is at $d_* e_1$. 
            In this case, we have  ${\rm supp}\, g \subset $  $\frac{1}{N_1} \times \frac1{\sqrt{K_*}} \times \cdots\times \frac1{\sqrt{K_*}}$-plate centred at $d_* e_1$ where $K_*:= (N_1/N_2)^2$. 
            Since $N_2/N_1 \le N_{1}^{-1/2}$, the function $g$ is already supported on a cap with radius $N_1^{-1/2}$. This means that we have \eqref{e:Step1} at the scale $K=N_1$ without the loss of $\mathcal{D}(K^{\frac12})$. 
            Hence, by virtue of \eqref{e:MorePrecise} which does not contain the loss\footnote{If one uses Theorem 2.2 in \cite{GZ} instead of \eqref{e:MorePrecise}, then the extra loss of $\log\, N_1$ will appear. } of $\log\, K$, the problem is now reduced to show that 
            $$
            \| E_{\mathbb{P}^{d-1} } \overline{h} \|_{L^p(B_{N_1^2})}
	        \le 
            C_\varepsilon N_2^\varepsilon
	        \big(
	        \sum_{\overline{\theta} \in \overline{\mathcal{C}}_{N_1^{-1}}} 
	        \| E_{\mathbb{P}^{d-1} } \overline{h}_{\overline{\theta}}  \|_{L^p( w_{B_{N_1^2}} )}^2 
	        \big)^\frac12 
            $$
            for all $\overline{h}$ supported on $[-N_2/N_1,N_2/N_1]^{d-1}$.  
            After the scaling, this is equivalent to 
            $$
            \| E_{\mathbb{P}^{d-1} } \widetilde{h} \|_{L^p(\mathcal{S})}
	        \le 
            C_\varepsilon N_2^\varepsilon
	        \big(
	        \sum_{\widetilde{\theta} \in \overline{\mathcal{C}}_{N_2^{-1}}} 
	        \| E_{\mathbb{P}^{d-1} } \widetilde{h}_{\widetilde{\theta}}  \|_{L^p( w_{\mathcal{S}} )}^2 
	        \big)^\frac12, 
            $$
            where $\widetilde{h}:=\overline{h}(N_2/N_1\cdot )$ and $\mathcal{S}$ is a slab with dimension ${ \underbrace{N_1N_2 \times \cdots \times N_1N_2}_{d-1} }\times N_2^2$ obtained by the scaling of $B_{N_1^2}$. 
            Since ${\rm supp}\, \widetilde{h} \subset [-1,1]^{d-1}$, this can be proved by chopping $\mathcal{S}$ into balls with radius $N_2^2$ and applying Bourgain--Demeter's $\ell^2$-decoupling inequality with $N_2^\varepsilon$-loss to each term. 
            \end{proof}

	\subsection{Shell Strichartz estimate: Proof of Theorem \ref{theorem:StrichartzShell}}\label{subsection3.5}
	It is now standard to derive the Strichartz estimate on torus from decoupling inequality, see \cite{Bour13} for instance. 
	We hence give here a sketch of the proof of the implication of Theorem \ref{theorem:StrichartzShell} from Theorem \ref{theorem:ShellDecoupling}. 
	In fact we make use of the local decoupling inequality Theorem \ref{cl:ShellDecoup}. 
	
	\begin{proof}[Proof of Theorem \ref{theorem:StrichartzShell}]
     It suffices to consider the case of $c_* =N_1$ and $p=\frac{2(d+1)}{d-1}$; the argument for other $c_*$ is parallel. Hence, ${\rm supp}\, \hat{\phi} \subset \mathcal{S}_{N_1,N_2}:= \{ k \in \mathbb{Z}^d : N_1 -1 \le |k| \le N_1 + 1\} \cap B_{N_2}$. 
	Let us denote the Fourier coefficient of $\phi$ by $\{a_k\}_{k \in \mathbb{Z}^d}$ so that 
	$$
	\phi(x) 
	= 
	\sum_{ k \in  \mathcal{S}_{N_1,N_2} } a_{k} e^{ i k \cdot x }.
	$$
	Correspondingly we define $g:[-2,2]^d\to \mathbb{C}$ by 
	$$
	g(\xi):= \sum_{ k \in  \mathcal{S}_{N_1,N_2}  } a_{ k } \delta_{\frac{k}{N_1}}(\xi),
	\;\;\; \xi \in [-2,2]^d,
	$$
	where $\delta_{\xi_0}$ denotes the Dirac delta\footnote{To be precise, we need to involve an approximation argument for Dirac delta by smooth function, see \cite{Bour13} for details.} at $\xi_0$ on $\mathbb{R}^d$.  
	If one writes $c_\theta = \frac{k}{N_1}$, then $\theta := [-\frac{1}{2N_1}, \frac{1}{2N_1}]^d + c_\theta \in \mathcal{C}_{N_1^{-1}}$. 
	Hence we may regard $g$ as 
	$$
	g(\xi):= \sum_{ \theta \in \mathcal{C}_{N_1^{-1}} } a_{ N_1 c_\theta } \delta_{c_\theta}(\xi)
	$$
	and hence 
	\begin{align*}
	E_{\mathbb{P}^d} g (x,t)
	&= 
	\sum_{ \theta \in \mathcal{C}_{N_1^{-1}} } a_{ N_1c_\theta }  e^{ i( x\cdot c_\theta + t|c_\theta|^2 ) } \\
	&= 
	\sum_{  k \in \mathcal{S}_{N_1,N_2} } a_{ k }  e^{ i( \frac{x}{N_1}\cdot k + \frac{t}{N_1^2}|k|^2 ) } \\
	&= 
	e^{i \frac{t}{N_1^2} \Delta} \phi(\frac{x}{N_1}).
	\end{align*}
	Hence we have that 
	\begin{align}\label{e:ExtensionTori}
	\big\| E_{\mathbb{P}^d} g \big\|_{ L^p( Q_{ N_1^2 } ) }	
	&= 
	\big\| e^{i \frac{t}{N_1^2} \Delta} \phi(\frac{x}{N_1}) \big\|_{ L^p( Q_{ N_1^2 } ) } 	
	= 
	N_1^{ \frac{2}{p} } N_1^{ \frac{d}{p} } \big\| e^{i t \Delta} \phi \big\|_{ L^p( [0,1]\times [0, N_1]^d  ) } 
	\end{align}
	where $Q_{N_1^2}:= [0, N_1^2]^{d+1}$.  
	Regarding the right-hand side, as observed in \cite{Bour13}, one indeed has that 
	$$
	\big\| e^{i t \Delta} \phi \big\|_{ L^p( [0,1]\times [0, N_1]^d ) }
	\sim 
	N_1^\frac{d}{p}\big\| e^{i t \Delta} \phi \big\|_{ L^p( \mathbb{T}^{d+1}) }
	$$
	because of the periodicity of $e^{i t \Delta} \phi$ with respect to $x$ variable. Hence \eqref{e:ExtensionTori} shows that 
\begin{equation}\label{e:ExtensionTori2}
\big\| E_{\mathbb{P}^d} g \big\|_{ L^p( Q_{ N_1^2 } ) }
\sim
N_1^{ \frac{2}{p} } N_1^{ \frac{2d}{p} } \big\| e^{i t \Delta} \phi \big\|_{ L^p(  \mathbb{T}^{d+1} ) }. 
\end{equation}
On the other hand, $g$ satisfies \eqref{e:Assumpg2} because of the assumption on the Fourier support of $\phi$ and hence we may apply our local decoupling inequality Corollary \ref{Cor:TwoScaleDecoup} to this input to see that 
	\begin{align*}
	&\big\| E_{\mathbb{P}^d} g \big\|_{ L^p( Q_{ N_1^2 } ) }
	\lesssim_\varepsilon
	N_2^\varepsilon 
	\big(
	\sum_{\theta \in \mathcal{C}_{N_1^{-1}}}
	\big\| E_{\mathbb{P}^d} g_\theta \big\|_{ L^p( w_{Q_{ N_1^2 }} ) }^2
	\big)^\frac12
	\end{align*}
	whilst we have from the definition of $g$ that 
	$$
	E_{\mathbb{P}^d} g_{\theta}(x,t)
	= 
	a_{N_1c_\theta} e^{ i( x\cdot c_\theta + t|c_\theta|^2 ) }. 
	$$
	Therefore we obtain that 
	\begin{align*}
	\big\| E_{\mathbb{P}^d} g \big\|_{ L^p( Q_{ N_1^2 } ) }
	&\lesssim_\varepsilon
	 N_2^\varepsilon 
	\bigg( \sum_{ \theta \in \mathcal{C}_{N_1^{-1}} } \big( |a_{N_1c_\theta}| w_{ Q_{ N_1^2 }}(\mathbb{R}^{d+1})^\frac{1}{p} \big)^2 \bigg)^\frac12
	\\
	&\sim
	N_2^\varepsilon  N_1^{\frac{2(d+1)}{p}}
	\big( \sum_{ \theta \in \mathcal{C}_{N_1^{-1}} }  |a_{N_1c_\theta}|^2 \big)^\frac12
	\end{align*}
	and hence conclude the proof by combining this with \eqref{e:ExtensionTori2}.

	\end{proof}

	\subsection{Sharpness of Theorem \ref{theorem:ShellDecoupling}}
	Let us prove the sharpness of the range $p\le \frac{2(d+1)}{d-1}$ in Theorem \ref{theorem:ShellDecoupling}. 
	Namely we prove that if one has 
	\begin{equation}\label{e:DecoupNoloss}
	\| f \|_{L^p(\mathbb{R}^{d+1})}
	\le C  
	\big( 
	\sum_{\theta \in \mathcal{C}_{N^{-1}}} 
	\| f_{\theta} \|_{L^p(\mathbb{R}^{d+1})}^2 
	\big)^\frac12
	\end{equation}
	for all $N\gg1$ and all $f$ satisfying the shell constraint \eqref{e:ShellDecoup}, then $p \le \frac{2(d+1)}{d-1}$ is necessary. 
	To this end we let 
	$$
	\mathcal{S}_N 
	:= 
	\{ (\xi,\tau) \in \mathbb{R}^{d+1}: |\tau - |\xi|^2| \le \frac{1}{100N^2},\; 1 \le |\xi| \le 1 + \frac{1}{10N^2}  \}
	$$
	which is an $N^{-2}$-neighborhood of translated  $d-1$ dimensional sphere $\mathbb{S}^{d-1}$, and let $\hat{f} = \1_{\mathcal{S}_N}$ which is a characteristic function on $\mathcal{S}_N$. 
	Then it is straightforward to check the condition \eqref{e:ShellDecoup} for such $f$. 
	To give a lower bound of the left hand side of \eqref{e:DecoupNoloss} we compute $f$ 
	\begin{align*}
	|f(x,t)|
	&= 
	\big|\int_{ | \tau - |\xi|^2| \ \le \frac{1}{100N^2},\; 1\le |\xi|\le 1 + \frac1{10N^2} } 
	e^{2\pi i (x\cdot \xi + t \tau)}\, d\xi d\tau\big|\\
	&= 
	\big|
	\int_{\xi: 1\le |\xi| \le 1+ \frac1{10N^2}} e^{2\pi ix\cdot \xi}
	\big(
	\int_{\tau: |\tau-|\xi|^2|\le \frac{1}{100N^2}}
	e^{2\pi i t(\tau - 1)}
	\, d \tau 
	\big)\, d\xi
	\big|. 
	\end{align*}
	One can explicitly compute the integral as 
	$$
	\int_{\tau: |\tau-|\xi|^2|\le \frac{1}{100N^2}}
	e^{2\pi i t(\tau - 1)}
	\, d\tau
	= 
	c t^{-1} \sin \big( \frac{ 2\pi t }{100N^2} \big)
	e^{2\pi i (|\xi|^2-1)}
	$$
	and hence 
	\begin{align*}
	|f(x,t)|
	&=
	\frac1{100N^2} 
	\big| 
	\frac{ \sin \big( \frac{ 2\pi t }{100N^2} \big) }{ \frac{ 2\pi t }{100N^2} }
	\int_{\xi: 1\le |\xi|\le 1 + \frac1{10N^2}} 
	e^{2\pi i (x\cdot \xi + |\xi|^2 -1)}\, d\xi 
	\big|\\
	&\ge 
	\frac1{100N^2} 
	\big| 
	\frac{ \sin \big( \frac{ 2\pi t }{100N^2} \big) }{ \frac{ 2\pi t }{100N^2} }
	\big|
	\frac{1}{N^2} \chi_{\{ |x| \le \frac1{10}\}}\\
	&\ge 
	C N^{-2} N^{-2} \chi_{\{|x| \le \frac1{10}, |t| \le \frac{N^2}{10} \}}.  
	\end{align*}
	This reveals a lower bound 
	$$
	\| f \|_{L^p(B_{N^2})} 
	\ge C N^{-4 + \frac2p}. 
	$$
	 On the other hand $\hat{f_\theta} = \1_{\theta \cap \mathcal{S}_N}$ and $\theta' := \theta \cap \mathcal{S}_N$ is an $\underbrace{ N^{-1}\times \cdots \times N^{-1}}_{d-1} \times N^{-2} \times N^{-2}$ slab for each $\theta \in \mathcal{C}_{N^{-1}}$.  Hence we see from Hausdorff--Young's inequality that 
	 $$
	 \| f_\theta \|_{L^p(\mathbb{R}^{d+1})}
	 \lesssim 
	 \big( N^{d-1} N^2N^2 \big)^{\frac1p-1}
	 $$
	 from which we derive the bound for the right hand side of \eqref{e:DecoupNoloss} 
	 $$
	 \big( 
	 \sum_{\theta \in \mathcal{C}_{N^{-1}}}
	 \| f_\theta \|_{L^p(\mathbb{R}^{d+1})}^2
	 \big)^\frac12
	 \sim 
	 N^{-(d+3)}  N^{\frac1p( d+3)} 
	 \big( 
	 \sharp \{ \theta \in \mathcal{C}_{N^{-1}}: \theta \cap \mathcal{S}_N \neq \emptyset\}
	 \big)^\frac12. 
	 $$
	 Since $\mathcal{S}_N$ is $N^{-2}$-neighborhood of $d-1$ dimensional sphere, we know $ \sharp \{ \theta \in \mathcal{C}_{N^{-1}}: \theta \cap \mathcal{S}_N \neq \emptyset\} \sim N^{d-1}$. 
	 Putting altogether with the assumption \eqref{e:DecoupNoloss} we conclude $p \le \frac{2(d+1)}{d-1}$.

	In addition, when $p= \frac{2(d+1)}{d-1}$, $\varepsilon$-derivative loss $N^{\varepsilon}$ on the right-hand side of \eqref{e:DecoupFat} cannot be removed. 
	Namely, there exists $f$ satisfying \eqref{e:ShellDecoup} such that
	\[
	\|f\|_{L^{\frac{2(d+1)}{d-1}}} \gtrsim (\log N)^{\frac{d-1}{2(d+1)}} \big( \sum_{\theta \in \mathcal{C}_{N^{-1}}} \|f_\theta\|_{L^p(\mathbb{R}^{d+1})}^2 \big)^\frac12.
	\]
	To see this, from the observation in Subsection \ref{subsection3.5}, it suffices to find $\{a_k\}_{k \in \mathcal{S}_N}$ such that
	\begin{equation}
	\label{est:NecessityofEpsilon}
	\| \sum_{k \in \tilde{\mathcal{S}}_N} a_k e^{i (k \cdot x + |k|^2 t)} \|_{L_{t,x}^{\frac{2(d+1)}{d-1}}(\T^{d+1})}
	\gtrsim (\log N)^{\frac{d-1}{2(d+1)}} \|a_k\|_{\ell^2}.
	\end{equation}
	Here $\tilde{\mathcal{S}}_N:= \{ k \in \mathbb{Z}^d : N -1 \le |k| \le N + 1\}$. We set
	\[
	a_k = 
	\begin{cases}
	1 \quad & \mathrm{if} \ k \in \{ (N, \overline{k}_1, \ldots, \overline{k}_{d-1}) \in \Z^{d} \, : \, 1 \leq \overline{k}_j \leq d^{-\frac12} N^{\frac12} \ (j=1,\ldots,d-1)\},\\
	0 \quad & \mathrm{otherwise}.
	\end{cases}
	\]
	Then, \eqref{est:NecessityofEpsilon} is written as the lower bound of the exponential sum
	\[
	\bigl\| \sum_{1 \leq \overline{k}_1 \leq d^{-\frac12} N^{\frac12}} \dots 
	\sum_{1 \leq \overline{k}_{d-1} \leq d^{-\frac12} N^{\frac12}}  e^{i (\overline{k} \cdot \overline{x} + |\overline{k}|^2 t)} \bigr\|_{L_{t,\overline{x}}^{\frac{2(d+1)}{d-1}}(\T^{d})}
	\gtrsim (\log N)^{\frac{d-1}{2(d+1)}} N^{\frac{d-1}{4}}.
	\]
	This can be shown via a number theoretic approach. See Theorem 13.6 in \cite{Dem}.
\section{Proof of Theorem~\ref{theorem:main}}
\label{section:ProofMainTheorem}
In this section, we prove Theorem \ref{theorem:main}. 
We will show that the first order system \eqref{Zakharov2} is locally well-posed in $H^s (\T^d) \times H^{s-\frac12}(\T^d)$ if $s > s_0$. 

We write
\[
\mathcal{J}_S[F](t) = -i \int_0^t e^{i(t-t')\Delta}F(t') d t', \quad 
\mathcal{J}_{W_{\pm}}[G](t) = i \int_0^t e^{\mp i(t-t')\langle \nabla \rangle}  G(t') d t,
\]
and rewrite the system \eqref{Zakharov2} in integral form: 
\begin{align}
u(t) &= e^{it\Delta} u_0  + \frac{1}{2} \mathcal{J}_S[(w + \overline{w})u](t),\\
w(t) & = e^{-it \langle \nabla \rangle} w_0 + \mathcal{J}_{W_{+}} \Bigl[\frac{\Delta}{\langle \nabla \rangle} (u \overline{u}) + \frac{1}{2 \langle \nabla \rangle} (w + \overline{w}) \Bigr](t).
\end{align}

The following bilinear estimates play a crucial role.
\begin{proposition}
\label{proposition:KeyBilinearEst}
Let $s_0$ be as defined in~\eqref{assumption:regularity} and $s>s_0$. Then there exist $b>\frac12$ and $\delta >0$ such that 
\begin{align}
& \|u w\|_{X_S^{s,b-1+ \delta}} + \|u \overline{w}\|_{X_S^{s,b-1+\delta}} \lesssim \|u \|_{X_S^{s,b}} \|w\|_{X_{W_{+}}^{s-\frac12,b}}, \label{est:prop1.3-1}\\
& \|u \overline{u}\|_{X_{W_{+}}^{s+\frac12,b-1 + \delta}} \lesssim \|u\|_{X_S^{s,b}}^2.\label{est:prop1.3-2}
\end{align}
\end{proposition}

\begin{proof}
By duality and dyadic decomposition of the space-time Fourier supports of the functions, 
to prove \eqref{est:prop1.3-1} and \eqref{est:prop1.3-2}, it suffices to prove that if $s > s_0$, there exists small $\delta >0$ such that
\begin{equation}\label{est:prop1.3-3}
\begin{split}
&\Bigl| \int u_{1} \overline{v}_{2} w_{3,\pm} dt dx \Bigr|\\
& \lesssim (L_1L_2 L_3)^{\frac12-\delta} \Bigl(\frac{\min(N_1,N_2)}{\max(N_1,N_2)} \Bigr)^{\frac12}N_{\min}^{s-\frac12} \|u_{1}\|_{L_{t,x}^2} \|v_{2}\|_{L_{t,x}^2} \|w_{3,\pm} \|_{L_{t,x}^2},
\end{split}
\end{equation}
where $u_1=  P_{N_1, L_1}^S u$, $v_2 = P_{N_2,L_2}^S v$, $w_{3,\pm} = P_{N_3,L_3}^{W_{\pm}} w$, and $N_{\min}= \min (N_1,N_2,N_3)$. 
A similar reduction can be found in \cite{BHHT09}.
Next, we do a case-by-case analysis.\\
(1) \underline{$N_3 \lesssim N_1\sim N_2$}: 
In this case, our goal is the following estimate:
\begin{equation}\label{est:prop1.3-8}
\Bigl| \int u_{1} \overline{v}_{2} w_{3,\pm} dt dx \Bigr| \lesssim (L_1L_2 L_3)^{\frac12-\delta} N_{3}^{s-\frac12} \|u_{1}\|_{L_{t,x}^2} \|v_{2}\|_{L_{t,x}^2} \|w_{3,\pm} \|_{L_{t,x}^2}.
\end{equation}
First, we notice that from Plancherel's theorem and the Cauchy--Schwarz inequality, we have
\begin{align}
\notag
\Bigl| \int u_{1} \overline{v}_{2} w_{3,\pm} dt dx \Bigr| & \sim 
\Bigl| \int \bigl( \widetilde{u}_{1} * \widetilde{w}_{3,\pm} \bigr) \overline{\widetilde{v}}_{2} d \tau d k \Bigr|\\
\notag
& = \Bigl| \int \int  \widetilde{u}_{1} (\tau-\tau', k-k')\widetilde{w}_{3,\pm}(\tau',k') d\tau' dk'  \overline{\widetilde{v}}_{2} (\tau,k) d \tau d k \Bigr|\\
\label{est:prop1.3-9}
& \lesssim L_{\min}^{\frac12}N_{3}^{\frac{d}{2}} \|u_1\|_{L_{t,x}^2}  \|v_2\|_{L_{t,x}^2} \|w_{3,\pm}\|_{L_{t,x}^2}.
\end{align}
Here $L_{\min} = \min (L_1,L_2,L_3)$. The estimate \eqref{est:prop1.3-9} immediately gives \eqref{est:prop1.3-8} if $L_{\max} := \max(L_1,L_2,L_3) \gtrsim N_3^4$. Hence, we may assume $L_{\max} \leq N_3^4$, which means that it is sufficient to show \eqref{est:prop1.3-8} with $\delta=0$. 
Using the definition of $s_0$, it suffices to prove that for any $\varepsilon>0$, we have the following estimate:
\begin{equation}\label{est:prop1.3-4}
\Bigl| \int u_{1} \overline{v}_{2} w_{3,\pm} dt dx \Bigr|\\
\lesssim (L_1L_2 L_3)^{\frac12} C(N_3) \|u_{1}\|_{L_{t,x}^2} \|v_{2}\|_{L_{t,x}^2} \|w_{3,\pm} \|_{L_{t,x}^2},
\end{equation}
where
\[
C(N_3) := 
\begin{cases}
N_{3}^{\varepsilon} \quad &(d=3),\\
N_3^{\frac{1}{4}+\varepsilon} & (d=4),\\
N_{3}^{\frac{d}{2} -2+ \varepsilon} & (d \geq 5).
\end{cases}
\]
By almost orthogonality and Fubini's theorem, it suffices to prove that, for arbitrary $c_i \in \R$ $(i=1,2,3)$, it holds that
\begin{equation}\label{est:prop1.3-5}
\Bigl| \int u_{1,c_1} \overline{v}_{2,c_2} w_{3,\pm,c_3} dt dx \Bigr|\\
 \lesssim C(N_3) \|u_{1,c_1}\|_{L_{t,x}^2} \|v_{2,c_2}\|_{L_{t,x}^2} \|w_{3,\pm,c_3} \|_{L_{t,x}^2},
\end{equation}
where
\begin{align}
\label{ass:prop1.3-support-01}
\supp \widetilde{u}_{1,c_1} & \subset \{(\tau,k) \in \R \times \Z^d \, : \, |\tau - |k|^2 -c_1| \leq 1, &  & k \in B_{N_3}(\xi_1)\},\\ 
\label{ass:prop1.3-support-02}
\supp \widetilde{v}_{2,c_2} & \subset \{(\tau,k) \in \R \times \Z^d \, : \, |\tau - |k|^2 -c_2| \leq 1, & & k \in B_{N_3}(\xi_2)\},\\ 
\label{ass:prop1.3-support-03}
\supp \widetilde{w}_{3, \pm,c_3} & \subset \{(\tau,k) \in \R \times \Z^d \, : \, |\tau \mp \langle k \rangle -c_3| \leq 1, &  \langle & k \rangle \lesssim N_3\}.
\end{align}
Here, for $j=1,2$, $B_{N_3 (\xi_j)}$ denotes the ball of radius $N_3$ and centre $\xi_j \in \R^d$ satisfying $|\xi_j| \sim N_j$.
We consider~\eqref{est:prop1.3-5}. By Plancherel's theorem, we have
\begin{align*}
\Bigl| \int u_{1,c_1} \overline{v}_{2,c_2} w_{3,\pm,c_3} dt dx \Bigr| & \sim 
\Bigl| \int \bigl( \widetilde{u}_{1,c_1} * \widetilde{w}_{3,\pm,c_3} \bigr) \overline{\widetilde{v}}_{2,c_2} d\tau d k \Bigr|\\
 = \Bigl| \iint & \widetilde{u}_{1,c_1} (\tau-\tau', k-k')\widetilde{w}_{3,\pm,c_3}(\tau',k') d\tau' dk'  \overline{\widetilde{v}}_{2,c_2} (\tau,k) d\tau d k \Bigr|.
\end{align*}
The support conditions on $\widetilde{u}_1$, $\widetilde{v}_2$, $\widetilde{w}_{3,\pm}$ imply that, in the above, we may assume
\[
|\tau - \tau' - |k-k'|^2 -c_1| \leq 1, \quad | \tau -|k|^2 -c_2 | \leq 1, \quad | \tau' \mp \langle k' \rangle -c_3| \leq 1.
\]
This implies
\begin{align*}
3& \geq \bigl|\tau - \tau' - |k-k'|^2 -c_1 - ( \tau -|k|^2-c_2)  + \tau' \mp \langle k' \rangle -c_3 \bigr|\\
& \geq \bigl| |k-k'|^2-|k|^2 +\widetilde{c} \pm \langle k' \rangle \bigr| ,
\end{align*}
where $\widetilde{c} = c_1-c_2+c_3$. Hence, from the spatial support condition of $\widetilde{w}_{3,\pm,c_3}$ in \eqref{ass:prop1.3-support-03}, we have $\bigl| |k-k'|^2-|k|^2 +\widetilde{c}   \, \bigr| \lesssim N_{3} \lesssim N_1$. 
This implies that if there exists $M_1 >0$ such that $|k|^2 = M_1 + \mathcal{O}(N_{1})$, we may find $M_2>0$ such that $|k-k'|^2 = M_2 + \mathcal{O}(N_{1})$. 
The support conditions imply $M_1 \sim M_2 \sim N_1^2$ and then, by the almost orthogonality, in addition to the above support conditions~\eqref{ass:prop1.3-support-01} and \eqref{ass:prop1.3-support-02}, we may assume that there exist $d_1 \sim d_2 \sim N_1$ such that
\begin{align*}
& \supp \widetilde{u}_{1,c_1} \subset \{(\tau,k) \in \R \times \Z^d \, : \, d_1- 1 \leq |k| \leq d_1 + 1\},\\
& \supp \widetilde{v}_{2,c_2} \subset \{(\tau,k) \in \R \times \Z^d \, : \, d_2- 1 \leq |k| \leq d_2 + 1\}.
\end{align*}
Therefore, by using Corollary~\ref{corollary:StrichartzShell}, we obtain
\begin{equation}\label{est:prop1.3-6}
\|u_{1,c_1}\|_{L_{t,x}^{\frac{2(d+1)}{d-1}}} \lesssim N_{3}^{\frac{\varepsilon}{2}} \|u_{1,c_1}\|_{L_{t,x}^2}, \quad 
\|v_{2,c_2}\|_{L_{t,x}^{\frac{2(d+1)}{d-1}}} \lesssim N_{3}^{\frac{\varepsilon}{2}} \|v_{2,c_2}\|_{L_{t,x}^2}.
\end{equation}
In the case $d=3$, these estimates immediately yield~\eqref{est:prop1.3-5} as follows:
\begin{align*}
\Bigl| \int u_{1,c_1} \overline{v}_{2,c_2} w_{3,\pm,c_3} dt dx \Bigr| & \leq 
\|u_{1,c_1}\|_{L_{t,x}^4} \|v_{2,c_2}\|_{L_{t,x}^4} \|w_{3,\pm,c_3}\|_{L_{t,x}^2}\\
& \lesssim N_{3}^{\varepsilon} \|u_{1,c_1}\|_{L_{t,x}^2} \|v_{2,c_2}\|_{L_{t,x}^2} \|w_{3,\pm,c_3} \|_{L_{t,x}^2}.
\end{align*}
In the case $d=4$, it follows from \eqref{est:StrichartzW} with $p=q=\frac{10}{3}$ that
\[
\|w_{3,\pm,c_3}\|_{L_{t,x}^{\frac{5}{2}}} \leq \|w_{3,\pm,c_3}\|_{L_{t,x}^{\frac{10}{3}}}^{\frac12}
\|w_{3,\pm,c_3}\|_{L_{t,x}^2}^{\frac12} \lesssim N_3^{\frac14} \|w_{3,\pm,c_3}\|_{L_{t,x}^{2}}.
\]
This estimate and \eqref{est:prop1.3-6} provide that
\begin{align*}
\Bigl| \int u_{1,c_1} \overline{v}_{2,c_2} w_{3,\pm,c_3} dt dx \Bigr| & \leq 
\|u_{1,c_1}\|_{L_{t,x}^{\frac{10}{3}}} \|v_{2,c_2}\|_{L_{t,x}^{\frac{10}{3}}} \|w_{3,\pm,c_3}\|_{L_{t,x}^{\frac{5}{2}}}\\
& \lesssim  N_3^{\frac{1}{4}+\varepsilon}\|u_{1,c_1}\|_{L_{t,x}^2} \|v_{2,c_2}\|_{L_{t,x}^2} \|w_{3,\pm,c_3} \|_{L_{t,x}^2}.
\end{align*}
Lastly, we consider the case $d \geq 5$. Let
\[
\Bigl( \frac{1}{q_w}, \frac{1}{p_w} \Bigr) = \Bigl( \frac{2}{d+1}, \frac{1}{2}- \frac{4}{d^2-1} \Bigr).
\]
Notice that $\frac{1}{q_w} = \frac{d-1}{2}(\frac12 - \frac{1}{p_w})$. 
Then, \eqref{est:StrichartzW} yields
\begin{equation}\label{est:prop1.3-7}
\|w_{3, \pm, c_3}\|_{L_t^{q_w} L_x^{p_w}} \lesssim N_3^{\frac{2}{d-1}} \|w_{3,\pm,c_3}\|_{L_{t,x}^2}.
\end{equation}
Consequently, by using Bernstein inequality, \eqref{est:prop1.3-6} and \eqref{est:prop1.3-7}, we get
\begin{align*}
\Bigl| \int u_{1,c_1} \overline{v}_{2,c_2} w_{3,\pm,c_3} dt dx \Bigr| & \leq 
\|u_{1,c_1}\|_{L_{t,x}^{\frac{2(d+1)}{d-1}}} \|v_{2,c_2}\|_{L_{t,x}^{\frac{2(d+1)}{d-1}}} \|w_{3,\pm,c_3}\|_{L_{t,x}^{\frac{d+1}{2}}}\\
& \lesssim N_{3}^{\varepsilon} \|u_{1,c_1}\|_{L_{t,x}^2} \|v_{2,c_2}\|_{L_{t,x}^2} N_{3}^{\frac{d}{2}-\frac{2 d}{d-1}} \|w_{3,\pm,c_3} \|_{L_t^{q_w} L_x^{p_w}}\\
& \lesssim N_{3}^{\frac{d}{2}-2+\varepsilon}
\|u_{1,c_1}\|_{L_{t,x}^2} \|v_{2,c_2}\|_{L_{t,x}^2} \|w_{3,\pm,c_3} \|_{L_{t,x}^2},
\end{align*}
which completes the proof of~\eqref{est:prop1.3-5}. 


(2) \underline{$\min(N_1,N_2) \ll  N_3$}: 
By symmetry, we assume that $N_2 \leq N_1$, namely, $N_2 \ll N_1 \sim N_3$. 
The estimate \eqref{est:prop1.3-3} is written as
\begin{equation}
\label{est:prop1.3-9.5}
\Bigl| \int u_{1} \overline{v}_{2} w_{3,\pm} dt dx \Bigr|
 \lesssim (L_1L_2 L_3)^{\frac12-\delta} N_1^{-\frac12}N_{2}^{s} \|u_{1}\|_{L_{t,x}^2} \|v_{2}\|_{L_{t,x}^2} \|w_{3,\pm} \|_{L_{t,x}^2}.
\end{equation}
By Plancherel's theorem, we have
\begin{align*}
\Bigl| \int u_{1} \overline{v}_{2} w_{3,\pm} dt dx \Bigr| & \sim 
\Bigl| \int \bigl( \widetilde{u}_{1} * \widetilde{w}_{3,\pm} \bigr) \overline{\widetilde{v}}_{2} d\tau d k \Bigr|\\
& = \Bigl| \int \int  \widetilde{u}_{1} (\tau-\tau', k-k')\widetilde{w}_{3,\pm}(\tau',k') d\tau' dk'  \overline{\widetilde{v}}_{2} (\tau,k) d\tau d k \Bigr|.
\end{align*}
This implies that we may assume
\begin{align*}
3L_{\max} & \geq \bigl|\tau - \tau' + |k-k'|^2 - ( \tau +|k|^2)  + \tau' \pm \langle k' \rangle \bigr|\\
& \geq \bigl| |k-k'|^2-|k|^2 \bigr| - \langle k' \rangle.
\end{align*}
Thus, because $N_2 \ll N_1 \sim N_3$, we have $L_{\max} \gtrsim N_1^2$. 
Let us first consider the case $L_3 = L_{\max} \gtrsim N_1^2$. 
It follows from Bernstein inequality and ~\eqref{est:StrichartzS} with $(q, p)=(4,\frac{2 d}{d-1})$ that
\begin{align*}
\Bigl| \int u_{1} \overline{v}_{2} w_{3,\pm} dt dx \Bigr| & 
\leq \|u_1 \|_{L_t^4 L_x^{\frac{2d}{d-1}}} \|v_2\|_{L_t^4 L_x^{2d}} \|w_3\|_{L_{t,x}^2}\\
& \lesssim (L_3 N_1^{-2})^{\frac12 - 3 \delta}\|u_1 \|_{L_t^4 L_x^{\frac{2d}{d-1}}} N_2^{\frac{d-2}{2}} \|v_2\|_{L_t^4 L_x^{\frac{2d}{d-1}}}\|w_3\|_{L_{t,x}^2}\\
& \lesssim (L_1 L_2 L_3)^{\frac12 - \delta}N_1^{-1 + 6 \delta +\varepsilon} N_2^{\frac{d-2}{2}}
\|u_{1}\|_{L_{t,x}^2} \|v_{2}\|_{L_{t,x}^2} \|w_{3,\pm} \|_{L_{t,x}^2}.
\end{align*}
This completes the proof of \eqref{est:prop1.3-9.5} in the case $L_3 = L_{\max}$.
We consider the remaining cases. Since the case $L_2 = L_{\max} \gtrsim N_1^2$ can be treated in a similar way, we only consider the case $L_1 = L_{\max} \gtrsim N_1^2$. 
First, we consider the case $d=3$. By H\"{o}lder's inequality, Sobolev embedding, \eqref{est:StrichartzW} and \eqref{est:StrichartzS}, we get
\begin{align*}
& \Bigl| \int u_{1} \overline{v}_{2} w_{3,\pm} dt dx \Bigr|\\
& \lesssim 
\|u_1\|_{L_{t,x}^2} \|v_2\|_{L_t^{\frac{10}{3}} L_x^{5}} \|w_{3,\pm}\|_{L_t^{5} L_x^{\frac{10}{3}}}\\
& \lesssim \|u_1\|_{L_{t,x}^2} N_2^{\frac{3}{10}} \|v_2\|_{L_{t,x}^{\frac{10}{3}}} L_3^{\frac12}N_3^{\frac{2}{5}} \|w_{3,\pm}\|_{L_{t,x}^2}\\
& \lesssim  (L_2L_3)^{\frac12} N_1^{\frac{2}{5}}  N_2^{\frac{3}{10}+\varepsilon} \|u_1\|_{L_{t,x}^2} \|v_2\|_{L_{t,x}^2} \|w_{3,\pm}\|_{L_{t,x}^2}\\
& \lesssim (L_1L_2 L_3)^{\frac12 - \delta} N_1^{-\frac35 + 6 \delta} N_2^{\frac{3}{10}+\varepsilon} \|u_1\|_{L_{t,x}^2} \|v_2\|_{L_{t,x}^2} \|w_{3,\pm}\|_{L_{t,x}^2},
\end{align*}
which completes the proof of \eqref{est:prop1.3-9.5} when $d=3$.
Next, we show \eqref{est:prop1.3-9.5} for the case $d \geq 4$. 
By almost orthogonality and Fubini's theorem, it is enough to prove that, for arbitrary $c_i \in \R$ $(i=2,3)$, it holds that
\begin{equation}
\label{est:prop1.3-10}
\Bigl| \int u_{1} \overline{v}_{2,c_2} w_{3,\pm,c_3} dt dx \Bigr|
 \lesssim L_1^{\frac12-3\delta} N_1^{-\frac12}N_{2}^{s} \|u_{1}\|_{L_{t,x}^2} \|v_{2}\|_{L_{t,x}^2} \|w_{3,\pm} \|_{L_{t,x}^2},
\end{equation}
where
\begin{align}
\label{ass:prop1.3-support-04}
\supp \widetilde{v}_{2,c_2} & \subset \{(\tau,k) \in \R \times \Z^d \, : \, |\tau - |k|^2 -c_2| \leq 1, \ |k| \sim N_2\},\\ 
\label{ass:prop1.3-support-05}
\supp \widetilde{w}_{3, \pm,c_3} & \subset \{(\tau,k) \in \R \times \Z^d \, : \, |\tau \mp \langle k \rangle -c_3| \leq 1, \ k \in B_{N_2}(\xi_3)\}
\end{align}
for some $\xi_3 \in \R^d$ which satisfies $|\xi_3| \sim N_3$. 
To obtain \eqref{est:prop1.3-10}, we will prove
\begin{equation}
\label{est:prop1.3-11}
\| v_{2,c_2} w_{3,\pm,c_3} \|_{L_{t,x}^2}
 \lesssim  N_1^{\frac{1}{d-1}} N_{2}^{\frac{d}{2} - \frac{d}{d-1}+ \varepsilon} \|v_{2,c_2}\|_{L_{t,x}^2} \|w_{3,\pm,c_3} \|_{L_{t,x}^2},
\end{equation}
where $v_2$ and $w_{3,\pm,c_3}$ satisfy \eqref{ass:prop1.3-support-04} and \eqref{ass:prop1.3-support-05}, respectively. It is straightforward to see that \eqref{est:prop1.3-11} yields \eqref{est:prop1.3-10}. 
Indeed, if \eqref{est:prop1.3-11} holds, since $L_1 \gtrsim N_1^2$, we have
\begin{align*}
\Bigl| \int u_{1} \overline{v}_{2,c_2} w_{3,\pm,c_3} dt dx \Bigr| & \leq \|u_1\|_{L_{t,x}^2} \| v_{2,c_2} w_{3,\pm,c_3} \|_{L_{t,x}^2}\\
& \lesssim L_1^{\frac12-3\delta} N_1^{-1 + 6 \delta} \|u_1\|_{L_{t,x}^2} \| v_{2,c_2} w_{3,\pm,c_3} \|_{L_{t,x}^2}\\
& \lesssim L_1^{\frac12-3\delta} N_1^{-\frac{d-2}{d-1} + 6 \delta}N_{2}^{\frac{d}{2} - \frac{d}{d-1}+ \varepsilon}
\|u_1\|_{L_{t,x}^2}\|v_{2,c_2}\|_{L_{t,x}^2} \|w_{3,\pm,c_3} \|_{L_{t,x}^2},
\end{align*}
which verifies \eqref{est:prop1.3-10} since $d \geq 4$. 
We turn to the proof of \eqref{est:prop1.3-11}. 
Let us assume $N_2 \gg 1$ since the case $N_2 \lesssim 1$ is easily handled by the Cauchy--Schwarz inequality. 
If $(\tau, k) \in \supp \widetilde{w}_{3, \pm,c_3}$, then \eqref{ass:prop1.3-support-05} gives
\begin{equation}
\label{est:prop1.3-12}
|\tau \mp \langle \xi_3 \rangle - c_3| \leq 2N_2.
\end{equation}
This means that the temporal frequency of $\widetilde{w}_{3, \pm,c_3}$ is confined to an interval of length $4N_2$. 
Hence, the standard almost orthogonality argument can be applied. We give a rough sketch of it.

Let a non-negative valued function $\psi_{N_2} \in \mathcal{S}(\R)$ satisfy 
\[
\supp \psi_{N_2} \subset [- 2N_2, 2 N_2], \qquad \sum_{m \in N_2 \Z} \psi_{N_2}(\tau-m) =1 \quad \mathrm{for} \ \mathrm{all} \ \tau \in \R.
\]
Here $N_2 \Z :=\{ N_2 \ell \, : \, \ell \in \Z\}$. 
For $m \in N_2 \Z$ and $v \in L^2(\R \times \T^d)$, define the operator $\F_{t,x}({\mathcal{P}_{N_2,m} v})(\tau,k)= \psi_{N_2} (\tau-m) \widetilde{v}(\tau,k)$. Then, clearly $\sum_{m \in N_2 \Z} \mathcal{P}_{N_2,m} v = v$. We observe
\begin{align*}
\| v_{2,c_2} w_{3,\pm,c_3} \|_{L_{t,x}^2} & \lesssim \Bigl( \sum_{m_1 \in N_2 \Z} \Bigl\|\mathcal{P}_{N_2,m_1} 
\Bigl( \bigl( \sum_{m_2 \in N_2 \Z} \mathcal{P}_{N_2,m_2} v_{2,c_2} \bigr) w_{3,\pm,c_3}\Bigr) \Bigr\|_{L_{t,x}^2}^2 \Bigr)^{\frac12}.
\end{align*}
We deduce from \eqref{est:prop1.3-12} that for each $m_1 \in N_2 \Z$ there exists the set $M_{N_2,m_1} \subset N_2 \Z$ such that
\begin{align*}
& \mathcal{P}_{N_2,m_1} \Bigl( \bigl( \sum_{m_2 \in M_{N_2,m_1}} \mathcal{P}_{N_2,m_2} v_{2,c_2} \bigr) w_{3,\pm,c_3}\Bigr) =
\mathcal{P}_{N_2,m_1} \Bigl( \bigl( \sum_{m_2 \in N_2 \Z} \mathcal{P}_{N_2,m_2} v_{2,c_2} \bigr) w_{3,\pm,c_3}\Bigr),\\
& \# M_{N_2,m_1} \sim 1, \qquad \Bigl( \sum_{m_1 \in N_2 \Z} \Bigl\|\sum_{m_2 \in M_{N_2,m_1}} \mathcal{P}_{N_2,m_2} v_{2,c_2}\Bigr\|_{L_{t,x}^2}^2 \Bigr)^{\frac12} \sim \|v_{2,c_2}\|_{L_{t,x}^2}.
\end{align*}
As a result, in addition to \eqref{ass:prop1.3-support-04}, we may assume that there exists $\tau' \in \R$ such that
\[
\supp \widetilde{v}_{2,c_2}  \subset \{(\tau,k) \in \R \times \Z^d \, : \, |\tau - |k|^2 -c_2| \leq 1, \ |k| \sim N_2, \ |\tau-\tau'| \leq N_2\},
\]

Let $(\tau,k) \in \supp \widetilde{v}_{2,c_2}$. Then, $|\tau - |k|^2 -c_2| \leq 1$ and $|\tau-\tau'| \leq N_2$ give
\[
| |k|^2 -\tau' + c_2| \leq 2 N_2.
\]
This and $|k| \sim N_2$ mean that there exists $c_* \sim N_2$ such that $||k|-c_*| \lesssim 1$. 
Consequently, we may utilize Corollary~\ref{corollary:StrichartzShell} to estimate ${v}_{2,c_2}$ as
\begin{equation}\label{est:prop1.3-13}
\|v_{2,c_2}\|_{L_{t,x}^{\frac{2(d+1)}{d-1}}} \lesssim N_2^{\varepsilon} \| v_{2,c_2}\|_{L_{t,x}^2}.
\end{equation}
Let
\[
\Bigl(\frac{1}{q_3}, \frac{1}{p_3} \Bigr) = \Bigl(\frac{1}{d+1}, \, \frac{1}{2} - \frac{2}{(d-1)(d+1)}\Bigr).
\]
Then, $\frac{1}{q_3} = \frac{d-1}{2}(\frac{1}{2} - \frac{1}{p_3})$ holds. 
It follows from \eqref{est:prop1.3-13} and Corollary~\ref{corollary:StrichartzW} that
\begin{align*}
\| v_{2,c_2} w_{3,\pm,c_3} \|_{L_{t,x}^2} & \leq \|v_{2,c_2}\|_{L_{t,x}^{\frac{2(d+1)}{d-1}}} \| w_{3,\pm,c_3}\|_{L_{t,x}^{q_3}}\\
& \lesssim N_2^{\varepsilon}\|v_{2,c_2}\|_{L_{t,x}^2} N_2^{\frac{d}{2} - \frac{d}{d-1}} \| w_{3,\pm,c_3}\|_{L_t^{q_3} L_x^{p_3}}\\
& \lesssim N_1^{\frac{1}{d-1}} N_2^{\frac{d}{2} - \frac{d}{d-1}+\varepsilon} \|v_{2,c_2}\|_{L_{t,x}^2} 
\| w_{3,\pm,c_3}\|_{L_{t,x}^2},
\end{align*}
which completes the proof of \eqref{est:prop1.3-11}.
\end{proof}

Now we turn to the proof of Theorem \ref{theorem:main}. 
Since the proof is quite standard, we give only a rough sketch of it. 
We refer to \cite{GTV97}, \cite{Kishi13} for more details.

We define
\begin{align*}
\Phi_S(u,w) &= e^{it\Delta} u_0  + \frac{1}{2} \mathcal{J}_S[(w + \overline{w})u](t),\\
\Phi_{W}(u,w) &= e^{-it \langle \nabla \rangle} w_0 + \mathcal{J}_{W_{+}} \Bigl[\frac{\Delta}{\langle \nabla \rangle} (u \overline{u}) + \frac{1}{2 \langle \nabla \rangle} (w + \overline{w}) \Bigr](t).
\end{align*}
Our task is to show that if $s > s_0$, there exist $b>\frac12$ and $0<T<1$ such that
\begin{equation}
\label{mapping:Contraction}
\mathcal{M}(u,v) = (\Phi_S(u,w), \Phi_{W}(u,w))
\end{equation}
is a contraction mapping in the ball of suitable radius in $X^{s,b}_S(T) \times X^{s-\frac12,b}_{W_+}(T)$ centred at the origin.

We first introduce the well-known linear estimates. For the proof, we refer to \cite{GTV97}, Lemma 4.1 in \cite{Kishi13}.
\begin{lemma}
\label{lemma:X^{s,b}Linear}
Let $s, b \in \R$ and $0<T <1$. Then,
\begin{align*}
\|e^{it \Delta} u_0 \|_{X^{s,b}_S(T)} \lesssim \|u_0\|_{H^s(\T^d)},\\
\|e^{-it \langle \nabla \rangle} w_0 \|_{X^{s,b}_{W_+}(T)} \lesssim \|w_0\|_{H^s(\T^d)}.
\end{align*}
\end{lemma}
Next, we state the estimates for handling the Duhamel terms. We omit the proof. 
For the details, see e.g. Lemma 2.1 in \cite{GTV97} and Lemma 4.1 in \cite{Kishi13}.
\begin{lemma}
\label{lemma:X^{s,b}Duhamel}
Let $s \in \R$, $b>\frac12$, $0<T <1$, $\delta >0$ and $\psi \in C_0^{\infty}(\R)$ satisfy $\psi = 1$ on $[-1,1]$ and $\supp \psi \subset (-2 , 2)$. 
Define $\psi_T(t) = \psi(\frac{t}{T})$. Then,
\begin{align*}
\| \psi_T \mathcal{J}_S[F] \|_{X^{s,b}_S} & \lesssim T^{\delta} \| F\|_{X_S^{s,b-1+\delta}},\\
\| \psi_T \mathcal{J}_{W_{+}}[G] \|_{X^{s,b}_{W_{+}}} & \lesssim T^{\delta} \|  G\|_{X_{W_{+}}^{s,b-1+\delta}}.
\end{align*}
\end{lemma}
By combining Proposition \ref{proposition:KeyBilinearEst} and Lemma \ref{lemma:X^{s,b}Duhamel}, we obtain the following:
\begin{lemma}
\label{lemma:X^{s,b}Nonlinear}
Let $s > s_0$ and $0<T<1$. Then there exists $b >\frac12$ and $\delta>0$ such that
\begin{align}
\label{est:X^{s,b}NonlinearS}
& \|\mathcal{J}_S[(w + \overline{w})u] \|_{X^{s,b}_S(T)} 
\lesssim T^{\delta} \|u \|_{X_S^{s,b}(T)} \|w\|_{X_{W_{+}}^{s-\frac12,b}(T)},\\
\label{est:X^{s,b}NonlinearW}
& \Bigl\|\mathcal{J}_{W_{+}} \Bigl[\frac{\Delta}{\langle \nabla \rangle} (u \overline{u})  \Bigr] \Bigr\|_{X^{s,b}_{W_{+}}(T)} \lesssim T^{\delta} \|u\|_{X_S^{s,b}(T)}^2.
\end{align}
\end{lemma}
\begin{proof}
It follows from the definitions of $X_S^{s,b}(T)$ and $X_{W_{+}}^{s-\frac12,b}(T)$ that for $(u, w) \in X^{s,b}_S(T) \times X_{W_{+}}^{s-\frac12,b}(T)$, there exists $(U, W) \in X^{s,b}_S \times X_{W_{+}}^{s-\frac12,b}$ such that
\begin{align}
\label{est:uw=UW}
& (u(t), w(t)) = (U(t), W(t)) \quad \mathrm{if} \ 0 \leq t  < T,\\
\notag
& \|U\|_{X^{s,b}_S} \leq 2 \|u\|_{X^{s,b}_S(T)}, \ 
\|W\|_{X_{W_{+}}^{s-\frac12,b}} \leq 2 \|w\|_{X_{W_{+}}^{s-\frac12,b}(T)}.
\end{align}
We choose $b> \frac12$ and $\delta >0$ so that \eqref{est:prop1.3-1} and \eqref{est:prop1.3-2} hold. Then,
\begin{align}
\label{est:Bilinear_UW}
& \|U W\|_{X_S^{s,b-1+ \delta}} + \|U \overline{W}\|_{X_S^{s,b-1+\delta}} \lesssim \|U \|_{X_S^{s,b}} \|W\|_{X_{W_{+}}^{s-\frac12,b}},\\
\label{est:Bilinear_U^2}
& \|U \overline{U}\|_{X_{W_{+}}^{s+\frac12,b-1 + \delta}} \lesssim \|U\|_{X_S^{s,b}}^2.
\end{align}
Lemma 2.2 in \cite{GTV97} tells that \eqref{est:Bilinear_UW} and \eqref{est:Bilinear_U^2} imply 
\[
\mathcal{J}_S[(W + \overline{W})U]  \in C(\R; H^s(\T^d)), \quad 
\mathcal{J}_{W_{+}} \Bigl[\frac{\Delta}{\langle \nabla \rangle} (U \overline{U})  \Bigr] \in C(\R; H^s(\T^d)).
\]
Therefore, \eqref{est:uw=UW} tells that if $ 0 \leq t  < T$, it holds that
\[
\mathcal{J}_S[(w + \overline{w})u](t)  = \mathcal{J}_S[(W + \overline{W})U] (t),  \quad \mathcal{J}_{W_{+}} \Bigl[\frac{\Delta}{\langle \nabla \rangle} (u \overline{u})  \Bigr](t) = 
\mathcal{J}_{W_{+}} \Bigl[\frac{\Delta}{\langle \nabla \rangle} (U \overline{U})  \Bigr](t),
\]
and in particular 
\[
\mathcal{J}_S[(w + \overline{w})u] \in X^{s,b}_S(T), \quad 
\mathcal{J}_{W_{+}} \Bigl[\frac{\Delta}{\langle \nabla \rangle} (u \overline{u})  \Bigr] \in X_{W_{+}}^{s-\frac12,b}(T).
\]
Consequently, by using Lemma \ref{lemma:X^{s,b}Duhamel}, we see that
\begin{align*}
\|\mathcal{J}_S[(w + \overline{w})u] \|_{X^{s,b}_S(T)} 
& \leq
\| \psi_T \mathcal{J}_S[(W + \overline{W})U] \|_{X^{s,b}_S}\\
& \lesssim T^{\delta} \| U W + U \overline{W}\|_{X_S^{s,b-1+\delta}}\\
& \lesssim T^{\delta} \|U \|_{X_S^{s,b}} \|W\|_{X_{W_{+}}^{s-\frac12,b}} \\
& \lesssim T^{\delta} \|u \|_{X_S^{s,b}(T)} \|w\|_{X_{W_{+}}^{s-\frac12,b}(T)}.
\end{align*}
This completes the proof of \eqref{est:X^{s,b}NonlinearS}. 
\eqref{est:X^{s,b}NonlinearW} can be proved similarly.
\end{proof}
Now, by applying Lemmas \ref{lemma:X^{s,b}Linear} and \ref{lemma:X^{s,b}Nonlinear} with a suitable exponents $T$, $\delta$, we verify that \eqref{mapping:Contraction} is a contraction mapping in some ball of suitable 
radius in $X^{s,b}_S(T) \times X^{s-\frac12,b}_{W_+}(T)$ centred at the origin.\footnote{Notice that the term $\mathcal{J}_{W_{+}} [ \frac{1}{2 \langle \nabla \rangle} (w + \overline{w}) ]$ is easily handled.} We omit the details.

\section{Proof of Theorem \ref{theorem:NotC^2}}
\label{section:ProofNegativeResult}
In this section, we prove Theorem \ref{theorem:NotC^2}.
\begin{proof}[Proof of Theorem \ref{theorem:NotC^2}]
We follow the argument introduced by Bourgain \cite{Bou97}. See Section 6 in \cite{Bou97}. We also refer to \cite{Hol07} and \cite{Kishi13}. 

We only need to prove that for arbitrarily large $C \gg 1$ and any $T>0$, there exist $f$, $g \in C^{\infty}(\T^d)$ such that
\begin{equation}\label{goal:NotC^2}
\begin{split}
\sup_{0\leq t \leq T} & 
\Bigl\|\int_0^t e^{i(t-t')\Delta} \bigl( (e^{it' \Delta}f) (\cos (t'|\nabla|) g) \bigr) d t' \Bigr\|_{H^s(\T^d)} \\
& \geq C \|f\|_{H^s(\T^d)} \|g\|_{H^{s-\frac12}(\T^d)}.
\end{split}
\end{equation}
Let $N \gg 1$ and 
\[
f_N = N^{-s - \frac{d}{2}}\sum_{|k| \sim N} e^{i k \cdot x}, \quad 
g_N = N^{-(s - \frac12) - \frac{d}{2}} \sum_{|k| \sim N} \cos(k \cdot x).
\]
Then, it holds that $\|f_N \|_{H^s(\T^d)} \sim \|g_N \|_{H^{s-\frac12}(\T^d)} \sim 1$. 
We observe that
\begin{align*}
& \F_{x}\Bigl[ 
e^{i(t-t')\Delta} \bigl( (e^{it' \Delta}f) (\cos (t'|\nabla|) g) \bigr)
\Big](k)\\
& \sim  N^{- 2 s -d + \frac12}e^{-i(t-t')|k|^2} \sum_{\substack{|k'| \sim N \\ |k-k'| \sim N}}e^{-it'|k-k'|^2} \cos (t'|k'|).
\end{align*}
If $0\leq t' \leq t \leq \frac{1}{100 N^2}$ for any $|k| \sim |k'| \sim N$, we have
\[
\Re \bigl[ e^{-i(t-t')|k|^2} e^{-it'|k-k'|^2} \cos (t'|k'|) \bigr] \geq \frac12.
\]
Consequently, if $0 <t \leq \frac{1}{100 N^2}$, we obtain
\begin{align*}
& \Bigl\| \int_0^t e^{i(t-t')\Delta} \bigl( (e^{it' \Delta}f) (\cos (t'|\nabla|) g) \bigr) \Bigr\|_{H^s(\T^d)}^2\\
& \gtrsim \sum_{|k| \sim N} t^2 N^{2 s} \Bigl| \F_{x}\Bigl[ 
e^{i(t-t')\Delta} \bigl( (e^{it' \Delta}f) (\cos (t'|\nabla|) g) \bigr)
\Big] \Bigr|^2\\
& \gtrsim t^2 N^{- 2 s -2d + 1} \sum_{|k| \sim N}\Bigl( \sum_{\substack{|k'| \sim N \\ |k-k'| \sim N}} 1 \Bigr)^2 
\sim t^2 N^{- 2 s +d + 1}.
\end{align*}
For arbitrarily $C \gg 1$ and $T>0$, by choosing $N$ so that $N^{-s+ \frac{d-3}{2}} \gg C$ and $\frac{1}{100 N^2} \leq T$, we observe that 
\[
\sup_{0 \leq t \leq T}\Bigl\| \int_0^t e^{i(t-t')\Delta} \bigl( (e^{it' \Delta}f) (\cos (t'|\nabla|) g) \bigr) \Bigr\|_{H^s(\T^d)} 
\gtrsim N^{-s + \frac{d}{2}  -\frac{3}{2}},
\]
which implies \eqref{goal:NotC^2}. 
\end{proof}

\section{Appendix}
\subsection{A bonus on the number of integer solutions to certain Diophantine system}
It is nowadays standard that once one can prove decoupling type inequality for some submanifold and even integer $p$, then one can  derive a bound of integer solutions of a Diophantine system corresponding to the submanifold, see \cite{BDG,Dem,GLYZ21} and references therein. 
For example the celebrated resolution of Vinogradov's meanvalue theorem (conjecture) by Bourgain--Demeter--Guth \cite{BDG} took this route in which case the submanifold was the moment curve in $\mathbb{R}^n$.  
Let us consider a Diophantine system 
\begin{equation}\label{e:Dio}
\begin{cases}
x_1+y_1=z_1+w_1,\\
x_2+y_2=z_2+w_2,\\
x_3+y_3=z_3+w_3,\\
x_1^2+x_2^2+x_3^2
+
y_1^2+y_2^2+y_3^2
=
z_1^2+z_2^2+z_3^2
+
w_1^2+w_2^2+w_3^2
\end{cases}
\end{equation}
for variables $(x_1,x_2,x_3,y_1,y_2,y_3,z_1,z_2,z_3,w_1,w_2,w_3)$ and denote the  number of integer solutions in $[1,N]^{12}$ of \eqref{e:Dio} for fixed large $N$ by $\sharp_N$.  By taking the trivial solutions $(x_1, x_2, x_3) = (z_1, z_2, z_3)$ and $(y_1, y_2, y_3) = (w_1, w_2, w_3)$ one can see that $\sharp_N\gtrsim N^6$. 
Then the standard argument reveals that 
$$
\sharp_N
=
c
\| e^{it\Delta} f \|_{L^4_{t,x}(\mathbb{T}^{3+1})}^4
$$
for $\hat{f}(k) = \1_{[1,N]^{3}}(k)$. 
On the other hand, Bourgain--Demeter's $\ell^2$-decoupling ensures bounds for $\| e^{it\Delta} f \|_{L^p_{t,x}(\mathbb{T}^{3+1})}$ with $2\le p\le\frac{10}{3}$.
It is worth to mention that there is no even integer except $p=2$ in the range $2\le p\le\frac{10}{3}$ and hence the standard argument is not applicable to give a bound of $\sharp_N$ by $N^{6+\varepsilon}$.
On the hand Theorem \ref{theorem:StrichartzShell} is applicable for $p=4$ under the shell type constraint. 
Namely, letting $x=(x_1,x_2,x_3)$, $y=(y_1,y_2,y_3)$, $z=(z_1,z_2,z_3)$, $w=(w_1,w_2,w_3)$, we have the following: 
\begin{corollary}
Let $N\gg1$. Then the number of integer solutions to \eqref{e:Dio} satisfying $N - 10 \le |x|,|y|,|z|,|w| \le N$ is bounded by $C_\varepsilon N^{4+\varepsilon}$ for arbitrary small $\varepsilon>0$. 
\end{corollary}
Here notice that the number of the trivial solutions $x=z$ and $y=w$ under the constraint $N - 10 \le |x|,|y|,|z|,|w| \le N$ is comparable to $N^4$.

\subsection{A remark on the mixed-norm Strichartz estimates}\label{S6.2}
As is mentioned in the Introduction, Theorem~\ref{theorem:main} is sharp up to $\varepsilon$ loss when $d=3$ and $d \geq 5$. 
In the $d=4$ case, however, there might be a room to relax the regularity condition \eqref{assumption:regularity}. 
We here discuss a possible route to this problem by assuming the validity of the mixed norm Strichartz estimates on the torus. 
Namely we suppose the endpoint estimate 
\begin{equation}
\label{est:MixedStrichartz}
    \|e^{it \Delta} \phi\|_{L_t^2 L_x^{\frac{2d}{d-2}}(\T^{d+1})} \lesssim N^{\varepsilon}\|\phi\|_{L_x^2(\T^d)},
\end{equation}
holds for all $\varepsilon>0$ and ${\rm supp}\, \hat{\phi} \subset \{ k \in \mathbb{Z}^d: |k| \le N \} $, and then 
see how it implies the sharp well-posedness result up to $\varepsilon$ loss. 
For simplicity, let us assume $d=4$ and consider only the trilinear estimate \eqref{e:Trilinear}. 
A simple use of \eqref{est:MixedStrichartz} with $d=4$ implies
\begin{align*}
& \Big|
\int_{[- \pi, \pi] \times \mathbb{T}^4} e^{ it\Delta } \phi_1 \overline{e^{it\Delta}\phi_2} e^{\pm i t |\nabla|} \phi_3\, dtdx
\Big|\\
& \leq 
\|e^{ it\Delta } \phi_1 \|_{L_t^2 L_x^4} \|e^{ it\Delta } \phi_2 \|_{L_t^2 L_x^4} \|e^{\pm i t |\nabla|} \phi_3 \|_{L_t^{\infty} L_x^2}\\
& \lesssim 
N^{\varepsilon} 
\prod_{j=1}^3 \| \phi_j \|_{L^2(\mathbb{T}^d)},
\end{align*}
for ${\rm supp}\, \hat{\phi_j} \subset \{ |\xi| \sim N\} \cap \mathbb{Z}^d$. 
By employing this almost sharp estimate and the argument in this paper, we may show the local well-posedness of \eqref{Zakharov} with $d=4$ in $\mathcal{H}^{s, s-\frac12}$ if $s>\frac12$, which is an optimal result up to $\varepsilon$ loss. 
Hence, the mixed-norm Strichartz estimate \eqref{est:MixedStrichartz} or in general
\begin{equation}
\label{est:MixedStrichartzGene}
    \|e^{it \Delta} \phi\|_{L_t^q L_x^p(\T^{d+1})} \lesssim N^{\varepsilon}\|\phi\|_{L_x^2(\T^d)},
\end{equation}
for $\frac2q = d(\frac12 - \frac1p)$ with $d\ge3$, is valuable  problem from PDE point of view. 
Notice that, by an interpolation of \eqref{e:StrichartzS} and the trivial estimate $\|e^{it \Delta} \phi\|_{L_t^{\infty} L_x^2(\T^{d+1})} = \|\phi\|_{L_x^2(\T^d)}$, we have \eqref{est:MixedStrichartzGene} for $\frac2q = d(\frac12 - \frac1p)$ and $q \geq p$, and hence the case of $q< p$ is a main problem.
As far as we know, there is no sharp result for \eqref{est:MixedStrichartzGene} when $q<p$. In this direction, we mention the work due to Burq--G\'{e}rard--Tzvetkov \cite{BGT} where they established the mixed norm Strichartz estimates with a certain loss of regularity. 
After we uploaded this paper on arXiv, Dasu--Jung--Li--Madrid \cite{DJLM} announced that the mixed-norm $\ell^2$-decoupling inequality which implies \eqref{est:MixedStrichartzGene} is prohibited to go beyond Bourgain--Demeter's (pure-norm) decoupling inequality on the Strichartz line $\frac2q = d(\frac12 - \frac1p)$. In particular, one cannot expect the mixed-norm $\ell^2$-decoupling inequality which implies \eqref{est:MixedStrichartz}. As is mentioned in \cite{DJLM}, their counterexamples are not the one for the mixed norm Strichartz inequality \eqref{est:MixedStrichartzGene} and so the problem is still open.

\section*{Acknowledgments}
{This work was supported by Grant-in-Aid for JSPS Kakenhi grant numbers 21J00514, 22KJ0446 (Kinoshita), 19K03546, 19H01796 and 21K13806 (Nakamura). 
The second author is grateful to Sanghyuk Lee for his helpful advice on the decoupling theory.  
The authors would like to thank Mamoru Okamoto for his comment about the Strichartz estimate for the wave equation, and Sebastian Herr for suggesting the problem.
The authors also express their gratitude to Changkeun Oh for his comment on the work due to Guo--Zorin-Kranich.
Finally, the authors are grateful to the anonymous referees for their valuable suggestions which helped improve the article.}

\section*{Conflict of Interest}
On behalf of all authors, the corresponding author states that there is no conflict of interest. 

\section*{Data Availability Statement}
Data sharing not applicable to this article as no datasets were generated or analysed during the current study.

\end{document}